\documentclass{compositio}

\usepackage[dvipsnames]{xcolor}

\newcommand\myshade{85}
\colorlet{mylinkcolor}{Red}
\colorlet{mycitecolor}{Green}
\colorlet{myurlcolor}{Plum}

\usepackage{lmodern}

\usepackage[utf8]{inputenc}

\usepackage[inline,shortlabels]{enumitem}

\usepackage{amsmath,amssymb}
\usepackage{mathtools}
\usepackage{thmtools,thm-restate}
\usepackage{mathrsfs}
\usepackage{colonequals}
\usepackage[retainorgcmds]{IEEEtrantools}
\usepackage{tikz,tikz-cd}
\usetikzlibrary{decorations.pathmorphing}

\usepackage[pdfusetitle]{hyperref}
\usepackage[hyphenbreaks]{breakurl}

\hypersetup{
  linkcolor  = mylinkcolor!\myshade!black,
  citecolor  = mycitecolor!\myshade!black,
  urlcolor   = myurlcolor!\myshade!black,
  colorlinks = true,
}

\renewcommand{\Re}{\operatorname{Re}}
\renewcommand{\Im}{\operatorname{Im}}

\newcommand{\N}{\mathbf{N}}
\newcommand{\Z}{\mathbf{Z}}
\newcommand{\Q}{\mathbf{Q}}
\newcommand{\R}{\mathbf{R}}
\newcommand{\C}{\mathbf{C}}
\newcommand{\G}{\mathbf{G}}
\newcommand{\PP}{\mathbf{P}}
\newcommand{\A}{\mathbf{A}}
\newcommand{\LL}{\mathbf{L}}

\newcommand{\HS}{\underline{\mathrm{HS}}}
\newcommand{\Rr}{\mathrm{R}}
\newcommand{\suchthat}{\;\ifnum\currentgrouptype=16 \middle\fi\vert\;}

\DeclareMathOperator{\an}{an}
\DeclareMathOperator{\Aut}{Aut}
\DeclareMathOperator{\cc}{c}
\DeclareMathOperator{\ch}{ch}
\DeclareMathOperator{\colim}{colim}
\DeclareMathOperator{\Db}{D}
\DeclareMathOperator{\der}{der}
\DeclareMathOperator{\Diff}{Diff}
\DeclareMathOperator{\Eig}{Eig}

\DeclareMathOperator{\Ext}{Ext}
\DeclareMathOperator{\Fix}{Fix}
\DeclareMathOperator{\Gr}{Gr}
\DeclareMathOperator{\Hh}{H}
\DeclareMathOperator{\Hom}{Hom}
\DeclareMathOperator{\id}{id}
\DeclareMathOperator{\im}{im}
\DeclareMathOperator{\NS}{NS}
\DeclareMathOperator{\Oo}{O}
\DeclareMathOperator{\per}{per}
\DeclareMathOperator{\Pic}{Pic}
\DeclareMathOperator{\pr}{pr}
\DeclareMathOperator{\rk}{rk}

\DeclareMathOperator{\Spec}{Spec}
\DeclareMathOperator{\spin}{spin}
\DeclareMathOperator{\td}{td}
\DeclareMathOperator{\tr}{tr}
\newcommand{\Pics}{\mathscr{P}ic}
\newcommand{\Tw}{\mathscr{T}\!w}

\newcommand{\sTwr}{\mathrm{s}T\!\mathrm{w}}

\theoremstyle{plain}
\newtheorem{thm}{Theorem}[section]
\newtheorem{mainthm}{Theorem}

\newtheorem{prop}[thm]{Proposition}
\newtheorem{lem}[thm]{Lemma}
\newtheorem{cor}[thm]{Corollary}
\newtheorem{maincor}[mainthm]{Corollary}

\theoremstyle{definition}
\newtheorem{defn}[thm]{Definition}
\newtheorem{exmp}[thm]{Example}

\theoremstyle{remark}
\newtheorem{rem}[thm]{Remark}

\begin{document}

\title{Autoequivalences of twisted K3 surfaces}
\author{Emanuel Reinecke}
\email{\href{mailto:reinec@umich.edu}{reinec@umich.edu}}
\address{Department of Mathematics\\University of Michigan\\530 Church St\\Ann Arbor, MI 48109, USA}

\thanks{This material is based upon work supported by the National Science Foundation under Grant No.\ DMS-1501461 and by the Studienstiftung des deutschen Volkes.}

\begin{abstract}
  Derived equivalences of twisted K3 surfaces induce twisted Hodge isometries between them;
  that is, isomorphisms of their cohomologies which respect certain natural lattice structures and Hodge structures.
  We prove a criterion for when a given Hodge isometry arises in this way.
  In particular, we describe the image of the representation which associates to any autoequivalence of a twisted K3 surface its realization in cohomology:
  this image is a subgroup of index one or two in the group of all Hodge isometries of the twisted K3 surface.
  We show that both indices can occur.
\end{abstract}

\maketitle

\section{Introduction}\label{sect:intro}

Let $X$ be a K3 surface.
Although the diffeomorphism group of $X$ remains a mysterious object, it is partially understood via its cohomology representation.
To elaborate, the intersection pairing gives a lattice structure of signature $(3,19)$ on $\Hh^2(X,\Z)$.
Denote by $\Oo\bigl(\Hh^2(X,\Z)\bigr)$ the orthogonal group of this lattice.
We obtain a representation
\[ \Diff(X) \to \Oo\bigl(\Hh^2(X,\Z)\bigr) \]
by associating to each diffeomorphism its action on cohomology.

Results of Borcea and Donaldson \cite{MR849050,MR1066174} completely describe the image of this representation.
Let $V \subset \Hh^2(X,\R)$ be a maximal positive-definite subspace.
Composing an isometry $\varphi \in \Oo\bigl(\Hh^2(X,\Z)\bigr)$ with the orthogonal projection onto $V$ gives an isometry of $V$.
We call $\varphi$ \emph{signed} if it induces orientation-preserving isometries on one, or equivalently every, maximal positive-definite subspace of $\Hh^2(X,\R)$; cf.\ Definition~\ref{defn:or_pres}.
The image of $\Diff(X) \to \Oo\bigl(\Hh^2(X,\Z)\bigr)$ is the group of signed isometries, which we denote by $\Oo^+\bigl(\Hh^2(X,\Z)\bigr)$.

Let $\Db(X)$ be the bounded derived category of coherent sheaves on $X$.
Mirror symmetry suggests a similar picture for the group of autoequivalences of $\Db(X)$; cf.\ \cite[Conj.~5.4]{MR1866907}.
From here on, assume that $X$ is projective, in which case every autoequivalence of $\Db(X)$ is a Fourier--Mukai transform \cite[Thm.~2.2]{MR1465519}.

In \cite{MR893604}, Mukai finds an analog of the above cohomology representation for $\Aut\bigl(\Db(X)\bigr)$.
The Mukai lattice of a K3 surface $X$ is the pure Hodge structure of weight $2$
\[ \widetilde{\Hh}(X,\Z) \colonequals \Hh^0(X,\Z)(-1) \oplus \Hh^2(X,\Z) \oplus \Hh^4(X,\Z)(1), \]
together with a lattice structure of signature $(4,20)$ given by the Mukai pairing
\[ \langle \phi , \psi \rangle \colonequals \int_X (\phi_2 \wedge \psi_2 - \phi_0 \wedge \psi_4 - \phi_4 \wedge \psi_0). \]
Mukai constructs a representation
\[ \Aut\bigl(\Db(X)\bigr) \to \Oo\bigl(\widetilde{\Hh}(X,\Z)\bigr), \]
which associates to every Fourier--Mukai equivalence $\Phi_P \in \Aut\bigl(\Db(X)\bigr)$ a cohomological realization $\Phi^{\Hh}_P$.
Work of various authors \cite{MR2047679,ploog-thesis,MR2553878} shows that the image of this representation is the group $\Oo^+\bigl(\widetilde{\Hh}(X,\Z)\bigr)$ of signed Hodge isometries, which is defined as before.
See \cite[\S 16.3]{huybrechts2016lectures} for details.

\subsection*{The setup}

A \emph{twisted K3 surface} is a pair $(X,\alpha)$, consisting of a K3 surface $X$ and a torsion class $\alpha \in \Hh^2(X,\mathcal{O}^*_X)$.
From the viewpoint of mirror symmetry, it is natural to try to extend the previous results to twisted K3 surfaces; cf.\ \cite[Conj.~5.5.5]{MR2700538} and \cite[Conj.~4.9]{MR2179782}.
Huybrechts and Stellari \cite{MR2179782} generalize Mukai's construction to a twisted cohomology representation
\[ \Aut\bigl(\Db(X,\alpha)\bigr) \to \Oo\bigl(\widetilde{\Hh}(X,B,\Z)\bigr), \]
again using cohomological realizations of twisted Fourier--Mukai functors.
We briefly introduce the objects that appear in this representation and refer once more to \cite[Ch.~16]{huybrechts2016lectures} for details.

An $\alpha$-twisted sheaf on $X$ is, loosely speaking, a generalized sheaf which satisfies the cocycle condition up to a twist by $\alpha$.
We denote by $\Db(X,\alpha)$ the bounded derived category of the abelian category of $\alpha$-twisted coherent sheaves.
To define a twisted version of the Mukai lattice, we pick a $B$-field lift of $\alpha$;
that is, a class $B \in \Hh^2(X,\Q)$ such that $\exp\bigl(B^{0,2}\bigr) = \alpha$.
The \emph{$B$-twisted Hodge structure} $\widetilde{\Hh}(X,B,\Z)$ has the same underlying lattice as $\widetilde{\Hh}(X,\Z)$.
Its Hodge structure of weight $2$ is the image of Mukai's Hodge structure under the multiplication by $\exp(B) \colonequals 1 + B + \frac{B\,\wedge\,B}{2}$.
Two different $B$-field lifts of $\alpha$ induce isometric twisted Hodge structures.

\subsection*{Our results}

In this note, we describe the image of $\Aut\bigl(\Db(X,\alpha)\bigr) \to \Oo\bigl(\widetilde{\Hh}(X,B,\Z)\bigr)$.
As in the untwisted case, every autoequivalence is of Fourier--Mukai type \cite[Thm.~1.1]{MR2329310}.
Thus, it suffices to describe the cohomological realizations of twisted Fourier--Mukai equivalences.
We do this more generally for possibly different Fourier--Mukai partners.

For every K3 surface $X$ and $B \in \Hh^2(X,\Q)$, we can make natural choices of compatible orientations on all maximal positive-definite subspaces of $\widetilde{\Hh}(X,B,\R)$; see Example~\ref{exmp:pos_sign_str}.
An isometry between two twisted Hodge structures is called signed if the induced map on positive-definite subspaces is orientation-preserving; see Definition~\ref{defn:or_pres}.
Huybrechts and Stellari show that every signed Hodge isometry lifts to a twisted Fourier--Mukai equivalence \cite[Thm.~0.1]{MR2310257}.
We prove the converse.
\begin{restatable}{mainthm}{norhi}\label{thm:norhi}
  Let $(X,\alpha)$ and $(X',\alpha')$ be two complex, projective twisted K3 surfaces with $B$-field lifts $B \in \Hh^2(X,\Q)$ and $B' \in \Hh^2(X',\Q)$.
  Let $\Phi_P \colon \Db(X,\alpha) \xrightarrow{\sim} \Db(X',\alpha')$ be a twisted Fourier--Mukai equivalence.
  Then the induced Hodge isometry $\Phi^{\Hh}_P \colon \widetilde{\Hh}(X,B,\Z) \xrightarrow{\sim} \widetilde{\Hh}(X',B',\Z)$ is signed.
\end{restatable}
In particular, Theorem~\ref{thm:norhi} gives a new argument for the untwisted statement, which was treated in \cite{MR2553878}.
It avoids the passage to rigid-analytic varieties and instead relies on the twisted theory.
The proof of Theorem~\ref{thm:norhi}, which makes up the bulk of this paper, is outlined below.
It uses in an essential way that we allow for different Fourier--Mukai partners.
Theorem~\ref{thm:norhi} and \cite[Thm.~0.1]{MR2310257} together imply the following strong form of the derived global Torelli theorem for twisted K3 surfaces.
\begin{mainthm}\label{thm:stdgt}
  In the notation of Theorem~\ref{thm:norhi}, a Hodge isometry $\varphi \colon \widetilde{\Hh}(X,B,\Z) \xrightarrow{\sim} \widetilde{\Hh}(X',B',\Z)$ is the cohomological realization of a twisted Fourier--Mukai equivalence $\Phi_P \colon \Db(X,\alpha) \xrightarrow{\sim} \Db(X',\alpha')$ if and only if $\varphi$ is signed.
\end{mainthm}
As a corollary, we obtain a complete description of the image of the cohomology representation of $\Aut\bigl(\Db(X,\alpha)\bigr)$.
\begin{maincor}\label{cor:main}
  The image of $\Aut\bigl(\Db(X,\alpha)\bigr) \to \Oo\bigl(\widetilde{\Hh}(X,B,\Z)\bigr)$ is the subgroup of signed Hodge isometries $\Oo^+\bigl(\widetilde{\Hh}(X,B,\Z)\bigr)$.
\end{maincor}
In another direction, we discuss in \S\ref{sect:hodge_thy} the index of the image in Corollary~\ref{cor:main}.
In the untwisted setting, the index of $\Oo^+\bigl(\widetilde{\Hh}(X,\Z)\bigr)$ in $\Oo\bigl(\widetilde{\Hh}(X,\Z)\bigr)$ is always $2$ because $\id_{(\Hh^0 \oplus \Hh^4)(X,\Z)} \oplus -\id_{\Hh^2(X,\Z)}$ is a non-signed Hodge isometry; see Example~\ref{exmp:non-signed_Hodge_isometry}.
The twisted setting is more involved:
even though $\Oo^+\bigl(\widetilde{\Hh}(X,B,\Z)\bigr)$ is still of index at most $2$ in $\Oo\bigl(\widetilde{\Hh}(X,B,\Z)\bigr)$, the isometry $\id_{(\Hh^0 \oplus \Hh^4)(X,\Z)} \oplus -\id_{\Hh^2(X,\Z)}$ does in general not preserve twisted Hodge structures.

In Lemma~\ref{lem:hyp_pln_or_rev}, we formulate a sufficient condition for the existence of a non-signed Hodge isometry.
If $X$ has Picard rank at least $12$, this condition is satisfied and the index is $2$; see Corollary~\ref{cor:non-signed-large-Picard}.
However, Theorem~\ref{thm:signed} below shows that the index is often $1$ for small Picard rank.
The integral cohomology of any K3 surface is isometric to the K3 lattice $\Lambda$.
Thus, after the choice of a marking $\theta \colon \Hh^2(X,\Z) \xrightarrow{\sim} \Lambda$, any $B \in \Lambda_\Q$ induces a $B$-field on $X$.
\begin{restatable}{mainthm}{signed}\label{thm:signed}
  There exist $d \in \Z_{>0}$ and $B \in \Lambda_\Q$ such that every very general marked, primitively $d$-polarized K3 surface $X$ admits only signed $B$-twisted Hodge isometries.
\end{restatable}
In fact, we can do better:
in Theorem~\ref{thm:spinor_trivial}, we state an explicit, lattice-theoretic condition for when a given $B \in \Lambda_\Q$ satisfies the conclusion of Theorem~\ref{thm:signed}.
For concrete choices of $B$, many of the computations involved in checking this condition can be read off of \cite[Thm.~VII.12.1--4]{miranda2009embeddings}.

\subsection*{Outline of the proof of Theorem~\ref{thm:norhi}}
In \cite[Cor.~3.19]{MR2388559}, Theorem~\ref{thm:norhi} is proved under the additional assumption that $\Db(X,\alpha)$ and $\Db(X',\alpha')$ do not contain any object $E$ with $\Ext^*(E,E) \simeq \Hh^*(S^2,\C)$.
Such objects are called \emph{spherical}.
We reduce Theorem~\ref{thm:norhi} to this special case.
To do so, we find in \S\ref{sect:pf} a deformation of $(X',\alpha')$ whose very general fiber does not have any spherical objects.
After passage to an \'etale neighborhood on the base, we can construct deformations of $(X,\alpha)$ and $P \in \Db(X \times X', \alpha^{-1} \boxtimes \alpha')$ which match the deformation of $(X',\alpha')$; see \S\ref{sect:mod_compl} and \S\ref{subsect:deform_D}.
Finally, we observe that the induced Hodge isometries on the fibers of these deformation families stay locally constant (\S\ref{subsect:signed_deformation}) and use \cite[Cor.~3.19]{MR2388559} to conclude.

To obtain the deformation of $(X',\alpha')$, we use the moduli space $M_d[n]$ of primitively $d$-polarized K3 surfaces with level-$n$ structure.
For every $B' \in \Lambda_\Q$, there exists an $n \gg 0$ such that $B'$ determines a Brauer class on the universal family of $M_d[n]$ over an \'etale neighborhood of $X'$; see \S\ref{subsect:deformation-family}.
With a suitable choice of $B'$, the universal family together with this Brauer class is a deformation of $(X',\alpha')$.
However, for some $B'$ the twisted derived category of every K3 surface in this family contains spherical objects.
For instance, when $B' = 0$, the structure sheaf is a spherical object in the untwisted derived category of every K3 surface.
Thus, we cannot expect to find ourselves in the situation of \cite[Cor.~3.19]{MR2388559} after a single deformation process.

We get around this problem by performing two deformations.
Fix a level structure on $X'$ and $B' \in \Lambda_\Q$ such that the induced Brauer class on $X'$ is $\alpha'$.
First, we construct infinitely many $B_m \in \Lambda_\Q$ and infinitely many distinct reduced, effective divisors $Z_m$ on $M_d[n]$.
On each K3 surface which corresponds to a point of $Z_m$, the Brauer classes induced by $B'$ and $B_m$ are equal. 
Furthermore, the $B_m$ induce Brauer classes on the universal family for which a very general fiber does not have any spherical objects.
As the union $\bigcup Z_m$ is Zariski-dense in $M_d[n]$, it meets every \'etale neighborhood of the initial surface $X'$.
Thus, we can reach it using the deformation of $(X',\alpha')$.
From there, we perform a second deformation with some $B_m$ to arrive at a twisted K3 surface whose twisted derived category does not contain any spherical objects.
For details, see \S\ref{subsect:conclusion}.
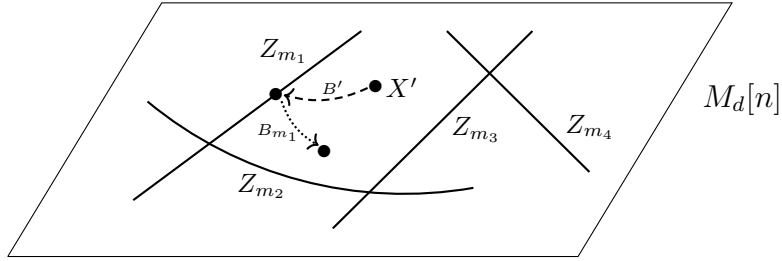
\begin{figure}[h]
  \centering
  \begin{tikzpicture}[scale=.75]
  \draw (.3,.5) -- (10.3,.5) -- (13,5) -- (3,5) -- (.3,.5);
  \node [right] at (12.3,3.3) {\large $M_d[n]$};
  \draw [thick] (2.5,1.5) -- (6.5,4.5);
  \node [above left] at (5.75,3.75) {\small $Z_{m_1}$};
  \draw [thick] (6,1) -- (9.5,4.5);
  \node [right] at (7.9,2.75) {\small $Z_{m_3}$};
  \draw [thick] (8,4.5) -- (10.5,2);
  \node [above right] at (9.9,2.4) {\small $Z_{m_4}$};
  \draw [thick] (2.75,3.25) arc [radius=7, start angle=230, end angle=280];
  \node at (4.75,1.75) {\small $Z_{m_2}$};
  \node at (6.75,3.5) {\large $\bullet$};
  \node [right] at (6.75,3.5) {\small $X'$};
  \draw [->, thick, bend left=20, densely dashed] (6.61,3.49) to (5.14,3.365);
  \node at (6,3.5) {\tiny $B'$};
  \node at (5,3.375) {\large $\bullet$};
  \draw [->, thick, bend right=20, densely dotted] (5.1,3.255) to (5.75,2.47);
  \node [left] at (5.6,2.7) {\tiny $B_{m_1}$};
  \node at (5.85,2.35) {\large $\bullet$};
\end{tikzpicture}
\end{figure}

Let us offer a small glimpse of the methods involved in finding the matching deformations of $(X,\alpha)$ and $P$, which generalize the methods of \cite{MR3429474}.
Let $\sTwr_{\mathcal{X}'/\C}$ be the coarse moduli space of simple, $\alpha'$-twisted complexes on $X'$; see \S\ref{subsection:gerbes}.
In \S\ref{subsection:FM_map}, we reinterpret $P$ as an open immersion $\bar{\mu}_P \colon X \hookrightarrow \sTwr_{\mathcal{X}'/\C}$.
Any given deformation of $(X',\alpha')$ over a base scheme $S$ formally produces a deformation of $\sTwr_{\mathcal{X}'/\C}$ over $S$.
As $\bar{\mu}_P$ is an open immersion, it has canonical deformations over each infinitesimal neighborhood in $S$.
After changing the Fourier--Mukai partners in an initial reduction (\S\ref{subsect:init_red}), we can use the algebraicity of a suitable Hilbert stack (\S\ref{subsect:Hilbert-stack}) to deform $\bar{\mu}_P$ over an \'etale neighborhood $U \to S$; see \S\ref{subsect:deform_D}.

\subsection*{Notation and conventions}\label{notation}

Throughout the paper, we work over the ground field $\C$.
We denote the category of perfect complexes on a scheme $X$ by $\Db(X)$;
the same notation is used for twisted schemes, algebraic spaces, algebraic stacks, complex manifolds, and differentiable manifolds.
We denote a Fourier--Mukai transform with kernel $P$ by $\Phi_P$.

We now introduce some lattice-theoretic notation.
The precise definitions of all following objects can be found in, say, \cite[Ch.~14]{huybrechts2016lectures}.
A lattice $L$ is a free $\Z$-module of finite rank together with a non-degenerate symmetric, bilinear form with values in the integers.
We call $L$ even if $(x.x) \in 2\Z$ for all $x \in L$.
The dual lattice of $L$ is $L^\vee \colonequals \Hom(L,\Z)$.
The discriminant group of $L$ is $A_L \colonequals L^\vee/L$, which is equipped with a finite quadratic form $q_L \colon A_L \to \Q / 2\Z$ when $L$ is even.
We denote the orthogonal group of $L$ by $\Oo(L)$ and the kernel of the natural map $\Oo(L) \to \Oo(A_L)$ by $\Oo^\sharp(L)$.

For any $a \in \Z \smallsetminus \{0\}$, we denote the lattice of rank $1$ with intersection matrix $(a)$ by $\langle a \rangle$.
The K3 lattice is $\Lambda \colonequals E_8(-1)^{\oplus 2}\oplus U^{\oplus 3}$, a direct sum of hyperbolic planes $U$ and twists of the $E_8$-lattice.
We denote the canonical basis vectors of the three copies of $U$ by $e_i$ and $f_i$, for $i = 1,2,3$.
The extended K3 lattice $\widetilde{\Lambda} \colonequals \Lambda \oplus U$ contains an additional hyperbolic plane with basis vectors $e$ and $f$.
The period domain associated with $\Lambda$ is $D \colonequals \lbrace x \in \PP(\Lambda_\C) \suchthat (x.x) = 0,\, (x.\bar{x}) > 0 \rbrace$.

For any $d \in \Z_{>0}$, we obtain $d$-polarized versions of the above notions by intersecting with the orthogonal complement of the primitive integral class $\ell \colonequals e_1 + d\cdot f_1 \in \Lambda$.
Concretely, $\Lambda_d \colonequals \Lambda \cap \ell^\perp$, $\widetilde{\Lambda}_d \colonequals \widetilde{\Lambda} \cap \ell^\perp$, and $D_d \colonequals D \cap \PP(\ell^\perp)$ are the $d$-polarized K3 lattice, the $d$-polarized extended K3 lattice, and the $d$-polarized period domain, respectively.

We define the following ring structure on $\widetilde{\Lambda}$ to mimic the cup product on singular cohomology:
\[ (\lambda e + x + \mu f) \cdot (\lambda' e + x' + \mu' f) \colonequals \lambda \lambda' e + (\lambda x' + \lambda' x) + (\lambda \mu' - (x.x') + \lambda' \mu)f. \]
The ``exponential'' of a given $B \in \Lambda_\C$ is $\exp(B) \colonequals e + B - \frac{(B.B)}{2}f \in \widetilde{\Lambda}_\C$.
It acts on $\widetilde{\Lambda}_\C$ via multiplication.
For a subset $S \subseteq \widetilde{\Lambda}_\C$, let $S^B \colonequals (\exp B) (S)$.
This applies in particular to $\Lambda^B$, $\Lambda^B_d$, and $\Lambda^B_{d\C}$.
As $\exp B$ descends to an automorphism $\PP(\exp B)$ of $\PP(\widetilde{\Lambda}_\C)$, an analogous definition can be made for subsets $S \subseteq \PP(\widetilde{\Lambda}_\C)$.
For instance, $D^B_d = \PP(\exp B)(D_d)$.

\section{Hodge theory of twisted K3 surfaces}\label{sect:hodge_thy}

This section describes the Hodge structures and Hodge isometries that can arise on twisted K3 surfaces.
Since it is, by and large, independent from the rest of the paper, the reader who is mainly interested in the proof of Theorem~\ref{thm:norhi} may jump directly to \S\ref{sect:mod_compl} and refer back as needed.
We begin by reviewing the necessary lattice theory.

\subsection{Lattice-theoretic preliminaries}\label{subsect:lattice_theory}

Let $L$ be a lattice with intersection form $(\;.\;)$.
The real vector space $L_\R \colonequals L \otimes_\Z \R$ comes with an inner product whose signature we denote by $(s_+,s_-)$.
When $s_+ > 0$, the \emph{positively oriented Grassmannian}
\[ \Gr^+_{s_+}(L_\R) \colonequals \lbrace W \subset L_\R \text{ positive-definite, oriented subspace of dimension } s_+ \rbrace \]
has two connected components.
\begin{defn}
  A \emph{positive sign structure} on $L$ is a choice of a connected component of $\Gr^+_{s_+}(L_\R)$.
\end{defn}
In other words, given two positive-definite subspaces $W$, $W'$ of $L_\R$ of maximal dimension $s_+$, the orthogonal projection $W \to W'$ is an isomorphism.
A positive sign structure is a choice of orientation on each such subspace $W$ that is compatible with these isomorphisms.
\begin{exmp}\label{exmp:pos_sign_str}
  Let $(X,\alpha)$ be a twisted K3 surface.
  Let $B \in \Hh^2(X,\Q)$ be a $B$-field lift of $\alpha$.
  Choose a holomorphic two-form $\sigma \in \Hh^{2,0}(X)$ and a K\"ahler class $\omega \in \Hh^{1,1}(X)$.
  The \emph{natural positive sign structure} on $\widetilde{\Hh}(X,B,\Z)$ is the connected component of $\Gr^+_4\bigl(\widetilde{\Hh}(X,B,\R)\bigr)$ that contains the positive-definite, oriented subspace with ordered basis
  \[ \left(\Re\bigl(\exp(B)\sigma\bigr), \Im\bigl(\exp(B)\sigma\bigr), \Re\bigl(\exp(B + i\omega)\bigr), \Im\bigl(\exp(B + i\omega)\bigr)\right). \]
  Here, $\exp(\tau) \colonequals 1 + \tau + \frac{\tau \wedge \tau}{2}$ for any $\tau \in \widetilde{\Hh}(X,B,\C)$ and $\exp(B)$ acts on $\widetilde{\Hh}(X,B,\C)$ via the cup product.
  The natural positive sign structure is independent of the choices of $\sigma$ and $\omega$.
\end{exmp}
\begin{defn}\label{defn:or_pres}
  An isometry $\sigma \colon L \xrightarrow{\sim} M$ of lattices with chosen positive sign structures is \emph{signed} if $\sigma$ carries the positive sign structure of $L$ into the positive sign structure of $M$.
  It is \emph{non-signed} otherwise.

  Likewise, an isometry $\sigma \colon L \xrightarrow{\sim} L$ of a lattice $L$ is signed if and only if $\sigma_\R$ preserves the two connected components of $\Gr^+_{s_+}(L_\R)$.
  Set
  \[ \Oo^+(L) \colonequals \lbrace \sigma \in \Oo(L) \suchthat \sigma_\R \text{ is signed} \rbrace. \]
\end{defn}
In the literature, positive sign structures are also called \emph{orientations of the positive directions} and signed isometries are those that \emph{preserve the orientation of the positive directions}. 
The next example is crucial for \S\ref{subsect:init_red}.
\begin{exmp}\label{exmp:non-signed_Hodge_isometry}
  Let $X$ be a K3 surface with holomorphic form $\sigma \in \Hh^{2,0}(X)$, K\"ahler class $\omega \in \Hh^{1,1}(X)$, and $B$-field $B \in \Hh^2(X,\Q)$.
  Then the isometry
  \[ \id_{\Hh^0\oplus \Hh^4} \oplus -\id_{\Hh^2} \colon \widetilde{\Hh}(X,B,\Z) \xrightarrow{\sim} \widetilde{\Hh}(X,-B,\Z) \]
  relates the basis vectors of the natural positive-definite, oriented subspaces from Example~\ref{exmp:pos_sign_str} as follows:
  \[ \begin{array}{ll}
    \Re\bigl(\exp(B)\sigma\bigr) \mapsto -\Re\bigl(\exp(-B)\sigma\bigr), & \Im\bigl(\exp(B)\sigma\bigr) \mapsto -\Im\bigl(\exp(-B)\sigma\bigr), \\
    \Re\bigl(\exp(B + i\omega)\bigr) \mapsto \Re\bigl(\exp(-B + i\omega)\bigr), & \Im\bigl(\exp(B + i\omega)\bigr) \mapsto -\Im\bigl(\exp(-B + i\omega)\bigr).
  \end{array} \]
  Thus, $\id_{\Hh^0\oplus \Hh^4} \oplus -\id_{\Hh^2}$ is non-signed.
\end{exmp}
We now give an alternative description of $\Oo^+(L)$.
\begin{defn}
  Let $K$ be a field of characteristic $\neq 2$, $V$ a vector space over $K$, and $(\;.\;)$ a symmetric, bilinear form on $V$.
  For $v \in V$ with $(v.v) \neq 0$, let $\tau_v \colon w\mapsto w - 2\frac{(w.v)}{(v.v)}v$ be the associated reflection.
  Set
  \[ \spin(\tau_v) = (v.v) \in K^\times / (K^\times)^2. \]
  By the Cartan--Dieudonn\'{e} theorem, every isometry is a composition of reflections.
  We can extend $\spin$ uniquely to a homomorphism
  \[ \spin \colon \Oo(V) \to K^\times / (K^\times)^2, \]
  which we call the \emph{spinor norm};
  it is well-defined by \cite[\S 55]{MR1754311}.
\end{defn}
In case $K = \R$, we identify $\R^\times / (\R^\times)^2 \simeq \{ \pm 1 \}$.
Then for any lattice $L$,
\[ \Oo^+(L) = \{ \sigma \in \Oo(L) \suchthat \det(\sigma) \cdot \spin(\sigma_\R) = 1 \}. \]
Denote by $\Oo^\sharp(L) \subseteq \Oo(L)$ the group of isometries that act trivially on the discriminant group $A_L = L^\vee/L$.
We later need to check when $\Oo^\sharp(L) \subseteq \Oo^+(L)$ (cf.\ Theorem~\ref{thm:spinor_trivial}).
Proposition~\ref{prop:ind_preserv_sign} below reduces this to a purely local computation.
Before we can state the precise criterion, we need to introduce additional notation.

For all primes $p \in \N$, let
\[ \Gamma_{p,0} \colonequals \{ \pm 1 \} \times \Z^\times_p /(\Z^\times_p)^2 \subset \Gamma_p \colonequals \{ \pm 1 \} \times \Q^\times_p /(\Q^\times_p)^2. \]
Let furthermore
\[ \Gamma_{\A,0} \colonequals \prod_{p\text{ prime}} \Gamma_{p,0} \subset \Gamma_\A \colonequals \biggl\{ (d_p,s_p) \in \prod_{p\text{ prime}} \Gamma_p \;\bigg|\; (d_p,s_p) \in \Gamma_{p,0} \text{ for almost all }p \biggr\}. \]
The inclusions $\Oo(L) \subset \Oo(L_{\Q_p})$ induce maps
\[ (\det,\spin_p) \colon \Oo(L) \to \Gamma_p, \]
which combine to
\[ (\det,\spin_\A) \colon \Oo(L) \to \Gamma_\A. \]
For any prime $p$, set
\[ \Sigma^\sharp_p(L) \colonequals (\det,\spin_p)\bigl(\Oo^\sharp(L_{\Z_p})\bigr) \quad \text{and} \quad \Sigma^\sharp(L) \colonequals \prod_{p\text{ prime}} \Sigma^\sharp_p(L). \]
By \cite[Thm.~VII.12.1--4]{miranda2009embeddings}, $\Sigma^\sharp_p(L) \subseteq \Gamma_{p,0}$ and $\Sigma^\sharp(L) \subseteq \Gamma_{\A,0}$.
Put
\[ \Gamma_\Q \colonequals \{ \pm 1 \} \times \Q^\times / (\Q^\times)^2 \subset \Gamma_\A. \]
Let $\Gamma^+_\Q \subset \Gamma_\Q$ be the preimage of $\{ (-1,-1),(1,1) \}$ under the natural homomorphism
\[ \Gamma_\Q = \{ \pm 1 \} \times \Q^\times / (\Q^\times)^2 \to \{ \pm 1 \} \times \R^\times / (\R^\times)^2 \simeq \{ \pm 1 \} \times \{ \pm 1 \}. \]
Lastly, set
\[ \Gamma_0 \colonequals \Gamma_\Q \cap \Gamma_{\A,0} = \{ \pm 1 \} \times \lbrace \pm 1 \rbrace \]
and define
\[ \Gamma^+_0 \colonequals \Gamma^+_\Q \cap \Gamma_0 = \{ (-1,-1),(1,1) \rbrace \subset \Gamma_0 \quad \text{and} \quad \Sigma^\sharp_0(L) \colonequals \Sigma^\sharp(L) \cap \Gamma_0. \]
The next proposition is known to some experts.
\begin{prop}\label{prop:ind_preserv_sign}
  Let $L$ be a non-degenerate, indefinite, even lattice of rank at least 3.
  Then $\Oo^\sharp(L) \subseteq \Oo^+(L)$ if and only if $\Sigma^\sharp_0(L) \subseteq \Gamma^+_0$.
\end{prop}
\begin{proof}
  To show necessity, assume $\Oo^\sharp(L) \subseteq \Oo^+(L)$.
  Let $\bigl((d_p,s_p)\bigr)_p \in \Sigma^\sharp_0(L)$.
  By definition of $\Sigma^\sharp(L)$, there are $\sigma_p \in \Oo^\sharp(L_{\Z_p})$ such that $(\det,\spin_p)(\sigma_p) = (d_p,s_p)$.
  They induce in particular an isometry $A_{(\sigma_p)}$ of the discriminant form $A_L$.

  On the other hand, as $L$ is indefinite of rank at least 3, $(\det,\spin) \colon \Oo(L_\Q) \to \Gamma_\Q$ is surjective; cf.\ \cite[Satz~A]{MR0080101}.
  Thus, there is $\rho \in \Oo(L_\Q)$ with $(\det,\spin_\A)(\rho) = \bigl((d_p,s_p)\bigr)_p$.
  By the strong approximation theorem for spin groups \cite[Approximationssatz]{MR0082514}, we can find $\psi \in \Oo(L_\Q)$ such that $(\det,\spin_\A)(\psi) = (1,1)$,
  \[ \psi_{\Q_p}(L_{\Z_p}) = \rho^{-1}_{\Q_p} \circ \sigma_p (L_{\Z_p}) = \rho^{-1}_{\Q_p} (L_{\Z_p}) \]
  for all $p$, and the induced morphism on discriminant groups is $A_{(\psi_p)} = A_{(\rho^{-1}_{\Q_p} \circ \sigma_p)}$.

  Set $\varphi \colonequals \rho \circ \psi$.
  Then $\varphi (L_{\Z_p}) = L_{\Z_p}$ for all $p$, and so $\varphi \in \Oo(L)$.
  Moreover, $A_\varphi = A_{(\sigma_p)} = \id_{A_L}$ because $\sigma_p \in \Oo^\sharp(L_{\Z_p})$.
  Therefore, $\varphi \in \Oo^\sharp(L)$ with $(\det,\spin_\A)(\varphi) = (\det,\spin_\A)(\rho) \cdot (\det,\spin_\A)(\psi) = \bigl((d_p,s_p)\bigr)_p$.
  However, since $\Oo^\sharp(L) \subseteq \Oo^+(L)$, it follows that $(\det,\spin_\A)(\varphi) \notin \lbrace (-1,1),(1,-1) \rbrace$.
  Thus, $\bigl((d_p,s_p)\bigr)_p \in \Gamma^+_0$.

  For sufficiency, assume $\Sigma^\sharp_0(L) \subseteq \Gamma^+_0$.
  Let $\varphi \in \Oo^\sharp(L)$.
  Then $(\det,\spin_\A)(\varphi) \in \Sigma^\sharp(L) \cap \Gamma_\Q$.
  As $\Sigma^\sharp(L) \subseteq \Gamma_{\A,0}$, we have
  \[ \Sigma^\sharp(L) \cap \Gamma_\Q = \Sigma^\sharp(L) \cap \Gamma_{\A,0} \cap \Gamma_\Q = \Sigma^\sharp(L) \cap \Gamma_0 = \Sigma^\sharp_0(L) \subseteq \Gamma^+_0. \]
  Hence, $(\det,\spin)(\varphi_\R) \in \lbrace (-1,-1),(1,1) \rbrace$ and $\varphi \in \Oo^+(L)$.
\end{proof}

\subsection{Large Picard rank}
Twisted Hodge structures of K3 surfaces with Picard rank at least $12$ always admit a non-signed Hodge isometry.
\begin{lem}\label{lem:hyp_pln_or_rev}
Let $X$ be a K3 surface and $B \in \Hh^2(X,\Q)$.
Assume there exists a hyperbolic plane $U \subset \widetilde{\Hh}^{1,1}(X,B,\Z)$.
Then $\widetilde{\Hh}(X,B,\Z)$ admits a non-signed Hodge isometry.
\end{lem}
\begin{proof}
  As $\widetilde{\Hh}^{1,1}(X,B,\Z) = U \oplus U^\perp$, we find the non-signed isometry $\id_{U^\perp} \oplus -\id_U$ of $\widetilde{\Hh}^{1,1}(X,B,\Z)$.
  Since $\id_{U^\perp} \oplus -\id_U$ acts trivially on $A_{\widetilde{\Hh}^{1,1}(X,B,\Z)}$, it can be extended to a non-signed Hodge isometry of $\widetilde{\Hh}(X,B,\Z)$; see \cite[Cor.~1.5.2]{MR525944}.
\end{proof}
\begin{cor}\label{cor:non-signed-large-Picard}
Let $X$ be a K3 surface of Picard rank $\rho (X) \geq 12$.
Let $B \in \Hh^2(X,\Q)$.
Then $\widetilde{\Hh}(X,B,\Z)$ admits a non-signed Hodge isometry.
\end{cor}
\begin{proof}
If $\rho(X) \geq 12$, then $\NS(X) \subset \widetilde{\Hh}^{1,1}(X,B,\Z)$ contains a hyperbolic plane by the proof of \cite[Lem.~4.1]{MR1314742}.
\end{proof}
We will see now, however, that the converse of Lemma~\ref{lem:hyp_pln_or_rev} does not hold.
\begin{prop}\label{prop:latt_gen_pic_gr}
Let $L$ be an even, non-degenerate lattice of rank $3 \leq \rk L \leq 12$ and signature $(2,\rk L - 2)$.
Assume there exists $x \in L$ with $(x.x) = 0$.
Then we can find a K3 surface $X$ and $B \in \Hh^2(X,\Q)$ such that $\widetilde{\Hh}^{1,1}(X,B,\Z) \simeq L$.
\end{prop}
\begin{proof}
  As $\rk L \leq 12 = \frac{1}{2} \cdot \rk \widetilde{\Lambda}$, there exists a primitive embedding $L \hookrightarrow \widetilde{\Lambda}$ \cite[Thm.~1.12.4]{MR525944}.
  Since the signature of $L$ is $(2,\rk L - 2)$, we can choose $p \in \widetilde{\Lambda}_\C$ such that
  \[ (p.p) = 0 ,\quad (p.\bar{p}) > 0, \quad \text{and} \quad L \subseteq p^\perp. \]
  Using the Baire category theorem, we can even ensure that $L = p^\perp \cap \widetilde{\Lambda}$.
  This defines the following Hodge structure of K3 type on $\widetilde{\Lambda}$:
  \[ \widetilde{\Lambda}^{2,0} \colonequals \C \cdot p ,\quad \widetilde{\Lambda}^{0,2} \colonequals \C \cdot \bar{p}, \quad \text{and} \quad \widetilde{\Lambda}^{1,1} \colonequals (\widetilde{\Lambda}^{2,0} \oplus \widetilde{\Lambda}^{0,2})^\perp. \]

  Let $y \in L$ with $(x.y) \neq 0$.
  Then the sublattice spanned by $x$ and $(x.y) y -\frac{(y.y)}{2} x$ is a twisted hyperbolic plane with intersection matrix $\left(\begin{smallmatrix} 0 & (x.y)^2 \\ (x.y)^2 & 0 \end{smallmatrix}\right)$.
  Thus, \cite[Lem.~2.6]{huybrechts2017cubic} shows that there is a complex K3 surface $X$ and $B \in \Hh^2(X,\Q)$ for which $\widetilde{\Hh}(X,B,\Z)$ is isometric to the Hodge structure on $\widetilde{\Lambda}$ defined above.
  In particular, $\widetilde{\Hh}^{1,1}(X,B,\Z) \simeq L$.

  For some $a \in \Z_{>0}$, there is an immersion $\Z \cdot (a,a B,0) \oplus \NS(X) \oplus \Hh^4(X,\Z) \hookrightarrow \widetilde{\Hh}^{1,1}(X,B,\Z)$ of finite index.
  Therefore, we can find $M \in \Pic(X)$ with $(M)^2 > 0$.
  As a consequence of Grauert's ampleness criterion for surfaces, $X$ is projective.
\end{proof}
\begin{rem}
  Any even, non-degenerate lattice $L$ of rank $5 \leq \rk L \leq 12$ and signature $(2, \rk L -2 )$ satisfies the assumptions of Proposition~\ref{prop:latt_gen_pic_gr} by Meyer's theorem.
\end{rem}
\begin{exmp}
  By Proposition~\ref{prop:latt_gen_pic_gr}, there is a K3 surface $X$ and $B \in \Hh^2(X,\Q)$ such that $\widetilde{\Hh}^{1,1}(X,B,\Z) \simeq \langle 4 \rangle \oplus U(2)$.
  As $\langle 4 \rangle \oplus U(2)$ contains no $(+2)$-classes, there is no primitive embedding $U\hookrightarrow \widetilde{\Hh}^{1,1}(X,B,\Z)$.
  Nevertheless, the non-signed isometry $\id_{\langle 4 \rangle} \oplus -\id_{U(2)}$ of $\widetilde{\Hh}^{1,1}(X,B,\Z)$ acts trivially on $A_{\widetilde{\Hh}^{1,1}(X,B,\Z)}$ and therefore induces a Hodge isometry of $\widetilde{\Hh}(X,B,\Z)$.
\end{exmp}

\subsection{Small Picard rank}\label{subsect:small-Picard-rank}

Finally, we consider very general twisted K3 surfaces.
Denote the standard basis vectors of a hyperbolic plane in $\Lambda$ by $e_1$ and $f_1$.
Fix $d \in \Z_{>0}$ and set $\ell \colonequals e_1 + d\cdot f_1$ and $\Lambda_d \colonequals \Lambda \cap \ell^\perp$.

A \emph{marked, primitively $d$-polarized K3 surface} is a triple $(X,A,\theta)$, consisting of a projective K3 surface $X$, an ample line bundle $A$, and a marking $\theta \colon \Hh^2(X,\Z) \xrightarrow{\sim} \Lambda$ such that $\theta\bigl(\cc_1(A)\bigr) = \ell$.
From here on, we suppress the ample line bundle from the notation.
All marked, primitively $d$-polarized K3 surfaces are parameterized by a fine moduli space $N_d$, which is a $19$-dimensional complex Hausdorff manifold.
Thanks to the local and global Torelli theorem, we can describe $N_d$ explicitly using periods.
Let $D_d \colonequals \lbrace p \in \PP(\Lambda_{d\C}) \suchthat (p.p) = 0,\, (p.\bar{p}) > 0 \rbrace$ be the period domain associated with $\Lambda_d$.
By identifying $(X,\theta)$ with its period $\bigl[\theta\bigl(\Hh^{2,0}(X)\bigr)\bigr] \in \PP(\Lambda_{d\C})$, we get
\[ N_d = D_d \smallsetminus \bigcup_{\substack{\delta \in \Lambda_d \\ (\delta .\delta) = -2}} \PP(\delta^\perp) \]
inside $\PP(\Lambda_{d\C})$.
See for example \cite[Ch.~6]{huybrechts2016lectures} for details.

Now we bring twists into play.
Fix $B \in \Lambda_\Q$ and set $\Lambda^B_d \colonequals (\exp B)(\Lambda_d) \subset \widetilde{\Lambda}_\Q$.
For every marked K3 surface $(X,\theta)$, we obtain the twisted Hodge structure $\widetilde{\Hh}(X,\theta^{-1}(B),\Z)$.

The marking $\theta$ can be extended uniquely to an isomorphism $\widetilde{\theta} \colon \Hh^*(X,\Z) \xrightarrow{\sim} \widetilde{\Lambda}$ which maps $1 \in \Hh^0(X,\Z)$ to the additional basis vector $e \in \widetilde{\Lambda}$;
with the ring structure on $\widetilde{\Lambda}$ from \S 1, $\widetilde{\theta}$ is multiplicative.
Under this extended marking, $\widetilde{\Hh}(X,\theta^{-1}(B),\Z)$ is determined by the \emph{$B$-twisted period} $\PP(\exp B)\bigl(\bigl[\theta\bigl(\Hh^{2,0}(X)\bigr)\bigr]\bigr) \in \PP\bigl(\Lambda^B_{d\C}\bigr)$.
If the $B$-twisted period does not lie on any of the projective hyperplanes $\PP\bigl(\delta^\perp\bigr)$, for $\delta \in \Lambda^B_d \cap \widetilde{\Lambda}$, then $\widetilde{\Hh}^{1,1}(X,\theta^{-1}(B),\Z)$ is isometric to
\[ L^B_d \colonequals \bigl(\Lambda^B_{d\C}\bigr)^\perp \cap \widetilde{\Lambda}. \]

Write $B = \widetilde{B} + \eta_1 e_1 + \eta_2 f_1$ with $\bigl(\widetilde{B}.e_1\bigr) = \bigl(\widetilde{B}.f_1\bigr) = 0$ and $\eta_1, \eta_2 \in \Q$, and put $B' \colonequals B - \eta_1\ell$.
\begin{lem}\label{lem:inters_form_LBd}
  The intersection matrix of the lattice $L^B_d$ is given by
  \[ \begin{pmatrix}
    2b & c & a \\
    c & 2d & 0 \\
    a & 0 & 0
  \end{pmatrix}_{\textstyle \raisebox{3pt}{,}} \]
  where $a \colonequals \per(B') = \min \lbrace k \in \Z_{>0} \suchthat kB' \in \Lambda \rbrace$ is the period of $B'$, $b\colonequals \frac{(aB' . aB')}{2}$, and $c \colonequals (aB' . \ell) = a(\eta_2 - d\eta_1)$.
\end{lem}
\begin{proof}
  Denote the standard basis vectors of $U$ in the decomposition $\widetilde{\Lambda} = \Lambda \oplus U$ by $e$ and $f$.
  Then
  \begin{IEEEeqnarray*}{rCl}
    L^B_d & = & \bigl(\Lambda^B_{d\C}\bigr)^\perp \cap \widetilde{\Lambda} \\
    & = &  \exp B (\C \cdot e + \C \cdot \ell + \C \cdot f) \cap \widetilde{\Lambda} \\
    & = & \left(\C \cdot \left( e + B - \tfrac{(B.B)}{2}f \right) + \C \cdot (\ell - (B.\ell)f) + \C \cdot f \right) \cap \widetilde{\Lambda} \\
    & = & (\C \cdot (e + B') + \C \cdot \ell + \C \cdot f) \cap \widetilde{\Lambda} \\
    & = & \Z \cdot (ae + aB') + \Z \cdot \ell + \Z \cdot f.
  \end{IEEEeqnarray*}
  Using this basis, we obtain the intersection matrix in the statement.
\end{proof}
\begin{rem}\label{rem:possible_B}
  All possible combinations of $a \in \Z_{>0}$ and $b,c \in \Z$ can occur.
  Indeed, if $e_2$ and $f_2$ denote the standard basis vectors of a second hyperbolic plane in $\Lambda$, then a given set of such $a$, $b$, and $c$ is attained by $B \colonequals \frac{c \cdot f_1 + e_2 + b \cdot f_2}{a}$.
\end{rem}
In what follows, we say a property holds for a very general point of a complex manifold if it holds outside countably many closed submanifolds of positive codimension.
\begin{thm}\label{thm:spinor_trivial}
  Let $B \in \Lambda_\Q$.
  The following are equivalent:
  \begin{enumerate}[label={\upshape(\arabic*)}]
    \item\label{thm:spinor_trivial_Hodge} a very general point of $N_d$ corresponds to a marked K3 surface $(X,\theta)$ for which all twisted Hodge isometries $\varphi \colon \widetilde{H}(X,\theta^{-1}(B),\Z) \xrightarrow{\sim} \widetilde{H}(X,\theta^{-1}(B),\Z)$ are signed
    \item $\Sigma^\sharp_0\bigl(L^B_d\bigr) \subseteq \Gamma^+_0$, where $\Sigma^\sharp_0\bigl(L^B_d\bigr)$ and $\Gamma^+_0$ are the groups from \S\ref{subsect:lattice_theory}.
  \end{enumerate}
\end{thm}
\begin{proof}
  Let $\psi \in \Oo\bigl(\widetilde{\Lambda}\bigr)$.
  Then $\psi$ induces a linear map $\PP(\psi)$ on $\PP\bigl(\widetilde{\Lambda}_\C\bigr)$.
  We denote the set of fixed points of $\PP(\psi)$ by $\Fix\bigl(\PP(\psi)\bigr)$.

  Let $X$ be a primitively $d$-polarized K3 surface with marking $\theta \colon \Hh^2(X,\Z) \xrightarrow{\sim} \Lambda$.
  Let $p \colonequals \PP(\exp B)\bigl(\bigl[\theta\bigl(\Hh^{2,0}(X)\bigr)\bigr]\bigr) \in \PP\bigl(\Lambda^B_{d\C}\bigr)$ be its $B$-twisted period.
  The morphism
  \[ \widetilde{\Hh}(X,\theta^{-1}(B),\Z) \simeq \widetilde{\Lambda} \xrightarrow{\quad\psi\quad} \widetilde{\Lambda} \simeq \widetilde{\Hh}(X,\theta^{-1}(B),\Z) \]
  is a Hodge isometry if and only if $p \in \Fix\bigl(\PP(\psi)\bigr)$.

  Put $D^B_d \colonequals \PP(\exp B)(D_d)$ and $N^B_d \colonequals \PP(\exp B)(N_d)$.
  We have an eigenspace decomposition
  \[ \Fix\bigl(\PP(\psi)\bigr) = \bigsqcup_{\lambda} \PP\bigl(\Eig(\psi_\C;\lambda)\bigr), \]
  where $\lambda$ runs through all eigenvalues of $\psi_\C$.
  Hence, the locus of $B$-twisted periods for which $\psi$ is a Hodge isometry is $\bigsqcup \bigl(N^B_d \cap \PP\bigl(\Eig(\psi_\C;\lambda)\bigr)\bigr)$.
  As $D_d \subset \PP(\Lambda_{d\C})$ is a smooth quadric, $D^B_d \cap \PP\bigl(\Eig(\psi_\C;\lambda)\bigr) \subseteq D^B_d$ is of positive codimension exactly when $\Lambda^B_{d\C} \nsubseteq \Eig(\psi_\C;\lambda)$.
  Moreover, $N^B_d$ is the complement of countably many hyperplane sections in $D^B_d$ and $\Oo(\widetilde{\Lambda})$ is countable.
  Thus, \ref{thm:spinor_trivial_Hodge} holds if and only if $\Lambda^B_{d\C} \nsubseteq \Eig(\psi_\C;\lambda)$ for all non-signed isometries $\psi \in \Oo(\widetilde{\Lambda})$ and all eigenvalues $\lambda$ of $\psi_\C$.

  On the other hand, whenever $\Lambda^B_{d\C} \subseteq \Eig(\psi_\C;\lambda)$ for some $\psi$ and $\lambda$, we have $\lambda = \pm 1$ because $\psi \in \Oo\bigl(\widetilde{\Lambda}\bigr)$ and $\Lambda^B_{d\C}$ is non-isotropic.
  Since $-\id_{\widetilde{\Lambda}}$ is signed, \ref{thm:spinor_trivial_Hodge} holds if and only if every $\psi \in \Oo\bigl(\widetilde{\Lambda}\bigr)$ with $\psi_\C \negthickspace\restriction_{\Lambda^B_{d\C}} = \id_{\Lambda^B_{d\C}}$ is signed.
  By \cite[Cor.~1.5.2, Thm.~1.6.1]{MR525944},
  \[ \left\{ \psi\in \Oo(\widetilde{\Lambda}) \suchthat \psi_\C \negthickspace \restriction_{\Lambda^B_{d\C}} = \id_{\Lambda^B_{d\C}} \right\} = \Bigl\{ \psi\in \Oo(\widetilde{\Lambda}) \;\Big|\; \psi \negthickspace \restriction_{\Lambda^B_d \cap \widetilde{\Lambda}} = \id_{\Lambda^B_d \cap \widetilde{\Lambda}} \Bigr\} \simeq \Oo^\sharp(L^B_d). \]
  Therefore, \ref{thm:spinor_trivial_Hodge} is equivalent to $\Oo^\sharp(L^B_d) \subseteq \Oo^+(L^B_d)$.
  The statement now follows from Proposition~\ref{prop:ind_preserv_sign}.
\end{proof}
\begin{rem}
  For non-polarized K3 surfaces, Theorem~\ref{thm:spinor_trivial} has an even simpler counterpart.
  Let $B \in \Lambda_\Q$.
  Set $a \colonequals \per(B)$ and $b \colonequals \frac{(aB.aB)}{2}$.
  The lattice $L^B \colonequals \bigl(\Lambda^B\bigr)^\perp \cap \widetilde{\Lambda}$ has intersection matrix $\left(\begin{smallmatrix}2b & a \\ a & 0\end{smallmatrix}\right)$.

  Given a marked complex K3 surface $(X,\theta)$, we can still define the twisted Hodge structure $\widetilde{\Hh}(X,\theta^{-1}(B),\Z)$.
  When $(X,\theta)$ is very general, $\widetilde{\Hh}^{1,1}(X,\theta^{-1}(B),\Z) \simeq L^B$.
  Classical theory of isotropic binary quadratic forms (cf.\ e.g.\ \cite[Ch.~13.3]{MR522835}) shows that such an $X$ admits only signed Hodge isometries if and only if neither $2 \mid a$ nor $b \equiv 1 \bmod a$.
\end{rem}
For any prime $p$, let $\Sigma^\sharp_0\bigl(L^B_d \otimes_\Z \Z_p\bigr)$ be the preimage of $\Sigma^\sharp_p(L^B_d)$ under the natural map $\Gamma_0 \to \Gamma_{p,0}$.
As $\Sigma^\sharp_0\bigl(L^B_d\bigr) = \bigcap_p \Sigma^\sharp_0\bigl(L^B_d \otimes_\Z \Z_p\bigr)$, the criterion of Theorem~\ref{thm:spinor_trivial} reduces the question as to whether a very general marked, primitively $d$-polarized K3 surface admits a non-signed $B$-twisted Hodge isometry to the purely local computations of the $\Sigma^\sharp_0\bigl(L^B_d \otimes_\Z \Z_p\bigr)$.
Those have been carried out in \cite[Thm.~VII.12.1--4]{miranda2009embeddings} for all indefinite lattices of rank at least 3.
\signed*
\begin{proof}
  Since the invariant factors of the intersection matrix of $L^B_d$ from Lemma~\ref{lem:inters_form_LBd} are
  \[ g_1 = \gcd(a,2b,c,2d),\, g_2 = \frac{\gcd(a^2,ac,2ad,4bd-c^2)}{g_1} \text{, and } g_3 = \frac{2a^2d}{g_1\cdot g_2}, \]
  the discriminant group of $L^B_d$ is $A_{L^B_d} = \Z / g_1\Z \oplus \Z / g_2\Z \oplus \Z / g_3\Z$.

  If there is a prime $p$ such that $p \mid g_1$ and $p \equiv 3 \mod 4$, \cite[Thm.~VII.12.1]{miranda2009embeddings} implies that $\Sigma^\sharp_0\bigl(L^B_d \otimes_\Z \Z_p\bigr) = \lbrace (1,1) \rbrace$. 
  Hence, $\Sigma^\sharp_0\bigl(L^B_d\bigr) = \lbrace (1,1) \rbrace \subset \Gamma^+_0$.
  In that case, Theorem~\ref{thm:spinor_trivial} shows that a very general marked, primitively $d$-polarized K3 surface does not admit any non-signed twisted Hodge isometries.
  Remark~\ref{rem:possible_B} yields the assertion.
\end{proof}

\section{The moduli space of twisted perfect complexes}\label{sect:mod_compl}

The next two sections set up the deformation theory for the Fourier--Mukai kernel $P$ from Theorem~\ref{thm:norhi}.
First, we discuss some preliminaries.

\subsection{Twisted complexes and gerbes}\label{subsection:gerbes}

Let $X$ be a noetherian algebraic space and $\alpha \in \Hh^2(X,\G_m)$.
We recall one possible construction of the category of $\alpha$-twisted sheaves on $X$.
Choose a $\G_m$-gerbe $\mathcal{X}$ over $X$ whose associated cohomology class is $\alpha$.
The inertia stack of $\mathcal{X}$ induces a $\G_m$-action on every sheaf $\mathcal{F}$ on $\mathcal{X}$.
When $\mathcal{F}$ is quasi-coherent, it has a weight decomposition
\[ \mathcal{F} = \bigoplus_{m \in \Z} \mathcal{F}^{(m)}, \]
where $\G_m$ acts on $\mathcal{F}^{(m)}$ via the character $\lambda \mapsto \lambda^m$; cf.\ \cite[Exp.~I, Prop.~4.7.3]{MR0274458}.

Consequently, the category of quasi-coherent sheaves on $\mathcal{X}$ has an orthogonal decomposition into full subcategories of sheaves of pure weight.
Passing to derived categories, we obtain a similar orthogonal decomposition of $\Db(\mathcal{X})$ into subcategories of perfect complexes of pure weight.
An \emph{$\alpha$-twisted} (or \emph{$\mathcal{X}$-twisted}) quasi-coherent sheaf is a quasi-coherent sheaf $\mathcal{F}$ on $\mathcal{X}$ with $\mathcal{F} = \mathcal{F}^{(1)}$; similarly for perfect complexes.

Work of To\"{e}n--Vaqui\'{e} \cite{MR2493386} and To\"{e}n \cite{MR2957304} shows that twisted perfect complexes are parameterized by a locally geometric, locally finitely presented derived stack.
In order to avoid derived objects, we restrict to complexes without higher automorphism groups.
\begin{defn}\label{defn:moduli_complexes}
  Let $S$ be a noetherian scheme and $f \colon X \to S$ be a smooth and proper morphism of schemes.
  Let $\mathcal{X}$ be a $\G_m$-gerbe over $X$ whose associated cohomology class in $\Hh^2(X,\G_m)$ is torsion.
  We define $\Tw_{\mathcal{X}/S}$ to be the fibered category of $\mathcal{X}$-twisted perfect complexes:
  its fiber $\Tw_{\mathcal{X}/S}(T)$ over a scheme $T \to S$ consists of all perfect complexes $E \in \Db(\mathcal{X} \times T)$ such that
  \begin{enumerate}[label={\upshape(\roman*)}]
    \item $E$ has pure weight $1$ and
    \item\label{defn:moduli_complexes_univ_gluable} $\Ext^i(E_t,E_t) = 0$ for all $i < 0$ and all points $t \in T$, where $E_t$ denotes the fiber of $E$ over $t$.
  \end{enumerate}
\end{defn}
  As explained above, the $\Ext$ groups computed in the category of twisted perfect complexes and in the category of all perfect complexes on $\mathcal{X}$ are equal.
\begin{thm}[{(\cite[Rmk.~5.6]{MR2957304}, \cite[Cor.~3.21]{MR2493386})}]\label{thm:tw_mod_ex}
  The fibered category $\Tw_{\mathcal{X}/S}$ from Definition~\ref{defn:moduli_complexes} is a locally finitely presented algebraic stack.
  We call it the \emph{moduli space of $\mathcal{X}$-twisted complexes}.
\end{thm}
\begin{rem}\label{rem:G_m_mu_n}
  Strictly speaking, \cite[Rmk.~5.6]{MR2957304} mentions $\mu_n$-gerbes, not $\G_m$-gerbes.
  However, as the cohomology class associated to $\mathcal{X}$ is torsion, the category of (quasi-)coherent sheaves twisted by $\mathcal{X}$ is equivalent to the category of (quasi-)coherent sheaves twisted by a suitable $\mu_n$-gerbe, and likewise for perfect complexes; cf.\ \cite[Lem.~3.1.1.12]{MR2388554}.
\end{rem}
\begin{rem}
  The results in \cite[Rmk.~5.6]{MR2957304} and \cite[Cor.~3.21]{MR2493386} are only formulated for affine $S$.
  Nonetheless, when $S$ is arbitrary, we can choose an affine open cover $S = \bigcup S_i$ and bootstrap the statement for $\coprod \Tw_{\mathcal{X}_{S_i}/S_i}$ to that for $\Tw_{\mathcal{X}/S}$; cf.\ \cite[\href{http://stacks.math.columbia.edu/tag/06DC}{Tag~06DC}]{stacks-project}.
\end{rem}
\begin{rem}\label{rem:univ_gluable}
  In place of \ref{defn:moduli_complexes_univ_gluable} in Definition~\ref{defn:moduli_complexes}, \cite{MR2493386} demands that
  \[ \Ext^i(E_U,E_U) = 0 \]
  for all $i < 0$ and all affine morphisms $g \colon U \to T$.
  Yet, the following argument shows that \ref{defn:moduli_complexes_univ_gluable} implies this seemingly stronger condition.

  As $\Ext^i(E_U,E_U) = \Hh^i\bigl(U,\R f_{U,*} \R\mathcal{H}om(E_U,E_U)\bigr)$, it suffices to deduce from condition \ref{defn:moduli_complexes_univ_gluable} that $\R f_{U,*} \R\mathcal{H}om(E_U,E_U)$ is concentrated in non-negative degrees.
  Since $E$ is perfect, \cite[Exp.~I, Prop.~7.1.2]{SGA6} and base change for (untwisted) complexes along Tor independent maps show that
  \[ \R f_{U,*}\R\mathcal{H}om(E_U,E_U) = \R f_{U,*} \LL g^*_\mathcal{X}\R\mathcal{H}om(E,E) = \LL g^* \R f_{T,*} \R\mathcal{H}om(E,E). \]
  Via \cite[\href{http://stacks.math.columbia.edu/tag/0BCD}{Tag~0BCD}]{stacks-project}, the assertion reduces to $\R f_{T,*} \R\mathcal{H}om(E,E) \otimes^\LL \kappa(t) \in \mathrm{D}^{\geq 0}\bigl(\kappa(t)\bigr)$ for all $t \in T$.
  By the reverse argument, this is tantamount to $\Ext^i(E_t,E_t) = 0$ for all $i < 0$ and all $t \in T$.
\end{rem}
\begin{rem}\label{rem:diag_qc}
  The diagonal of $\Tw_{\mathcal{X}/S}$ is quasi-compact.
  This follows directly from \cite[\S 3.3]{MR2493386}.
\end{rem}
We focus on a substack of complexes that is fibered as a $\G_m$-gerbe over an algebraic space.
\begin{defn}
  A complex $E \in \Tw_{\mathcal{X}/S}(T)$ is called \emph{simple} if the natural map of fppf sheaves $\G_{a,T} \to \mathcal{E}nd(E)$ is an isomorphism.
  We denote the full subcategory of $\Tw_{\mathcal{X}/S}$ whose objects are all simple complexes by $s\Tw_{\mathcal{X}/S}$.
\end{defn}
By the infinitesimal lifting criterion, the inclusion $s\Tw_{\mathcal{X}/S} \subset \Tw_{\mathcal{X}/S}$ is \'etale and hence an open immersion.
Let $\mathcal{I} \to s\Tw_{\mathcal{X}/S}$ be the inertia stack of $s\Tw_{\mathcal{X}/S}$.
As the natural map $\G_m \times s\Tw_{\mathcal{X}/S} \to \mathcal{I}$ is an isomorphism, rigidification \cite[Thm.~5.1.5]{MR2007376} makes $s\Tw_{\mathcal{X}/S}$ into a $\G_m$-gerbe $s\Tw_{\mathcal{X}/S} \to \sTwr_{\mathcal{X}/S}$ over an algebraic space $\sTwr_{\mathcal{X}/S}$, the coarse moduli space for $s\Tw_{\mathcal{X}/S}$; cf.\ also \cite[Cor.~3.22]{MR2493386}.

\subsection{Fourier--Mukai kernels as open immersions}\label{subsection:FM_map}

Let $(X,\alpha)$ and $(X',\alpha')$ be two smooth, projective twisted varieties.
Let $\Phi_P \colon \Db (X,\alpha) \to \Db(X',\alpha')$ be a fully faithful twisted Fourier--Mukai transform.
We generalize \cite[\S 5]{MR3429474} to interpret the Fourier--Mukai kernel $P$ as an open immersion from $X$ into the moduli space of $\alpha'$-twisted perfect complexes on $X'$.

First, choose $\G_m$-gerbes $\mathcal{X} \to X$ and $\mathcal{X}' \to X'$ that represent $\alpha^{-1}$ and $\alpha'$.
Just as in \S\ref{subsection:gerbes}, $P$ can be identified with a perfect complex of weight $(1,1)$ on the $\G_m \times \G_m$-gerbe $\mathcal{X} \times \mathcal{X}' \to X \times X'$.
\begin{lem}\label{lem:Hochschild_cohom}
  Let $\pi$ and $\pi'$ denote the projections from $X \times X$ and $X \times X'$ onto the first factor, respectively.
  Then convolution with $P$ gives a quasi-isomorphism
  \[ \gamma \colon \R\pi_* \R\mathcal{H}om_{\mathcal{X} \times \mathcal{X}}\bigl(\mathcal{O}_{\Delta_\mathcal{X}},\mathcal{O}_{\Delta_\mathcal{X}}\bigr) \to \R\pi'_* \R\mathcal{H}om_{\mathcal{X} \times \mathcal{X}'}(P,P) \]
  of complexes on $X$.
\end{lem}
Here, we regard $\mathcal{O}_{\Delta_\mathcal{X}}$ as a sheaf of weight $(1,-1)$ on the $\G_m \times \G_m$-gerbe $\mathcal{X} \times \mathcal{X} \to X\times X$.
Since both internal Hom complexes are of weight $(0,0)$, we identify them via derived pushforward with complexes on the underlying coarse spaces.
\begin{proof}
  It is enough to prove that for each closed point $x \in X$, the induced morphism on fibers $\gamma_x$ is an isomorphism.
  Fix $x \in X$ closed.
  Let $\iota_x \colon \mathcal{X}_x \hookrightarrow \mathcal{X}$ be the inclusion of the fiber of $\mathcal{X}$ over $x$.
  By base change along Tor independent maps, $\gamma_x$ is the convolution
  \[ \R\!\Hom_{\mathcal{X}_x \times \mathcal{X}}\bigl(\LL(\iota_x \times \id_\mathcal{X})^* \mathcal{O}_{\Delta_\mathcal{X}},\LL(\iota_x \times \id_\mathcal{X})^* \mathcal{O}_{\Delta_\mathcal{X}}\bigr) \to \R\!\Hom_{\mathcal{X}_x \times \mathcal{X}'}\bigl(\LL(\iota_x \times \id_{\mathcal{X}'})^* P,\LL(\iota_x \times \id_{\mathcal{X}'})^* P\bigr). \]

  The character $\G_m \times \G_m \to \G_m$ given by $(x,y) \mapsto x^{-1}y$ induces by surjective functoriality a $\G_m$-gerbe $\mathcal{X}_x \wedge \mathcal{X} \to \{x\} \times X$ together with a morphism $\mathcal{X}_x \times \mathcal{X} \to \mathcal{X}_x \wedge \mathcal{X}$, which is the rigidification along the diagonal of $\G_m \times \G_m$ \cite[\S IV.2.3.18]{MR0344253}.
  Similarly, the character $\G_m \times \G_m \to \G_m$ given by $(x,y) \mapsto xy$ produces $\mathcal{X}_x \times \mathcal{X}' \to \mathcal{X}_x \wedge \mathcal{X}'$, the rigidification along the antidiagonal of $\G_m \times \G_m$.
  Combining both characters, we obtain $\mathcal{X}_x \wedge (\mathcal{X} \times \mathcal{X}') \simeq (\mathcal{X}_x \wedge \mathcal{X}) \times (\mathcal{X}_x \wedge \mathcal{X}')$.
  Everything is summarized in the diagram below:
  \[ \begin{tikzcd}
      & \mathcal{X}_x \times \mathcal{X} \times \mathcal{X}' \arrow[ld, "\pr_{12}"'] \arrow[d, "p''"] \arrow[rd, "\pr_{13}"] & \\
      \mathcal{X}_x \times \mathcal{X} \arrow[d, "p"'] & (\mathcal{X}_x \wedge \mathcal{X}) \times (\mathcal{X}_x \wedge \mathcal{X}') \arrow[ld, "\pr_1"'] \arrow[rd, "\pr_2"] & \mathcal{X}_x \times \mathcal{X}' \arrow[d, "p'"] \\
      \mathcal{X}_x \wedge \mathcal{X} && \mathcal{X}_x \wedge \mathcal{X}'.
  \end{tikzcd} \]

  Since the gerbe $\mathcal{X}_x \to \Spec \C$ has a section, $\mathcal{X}_x \wedge \mathcal{X}$ and $\mathcal{X}_x \wedge \mathcal{X}'$ still represent $\alpha^{-1}$ and $\alpha'$, respectively \cite[\S IV.3.3.2]{MR0344253}.
  The kernel $\R p''_*\LL\!\pr^*_{23}P$ induces $\Phi_P$ on $(\mathcal{X}_x \wedge \mathcal{X}) \times (\mathcal{X}_x \wedge \mathcal{X}')$.
  Moreover, $\LL(\iota_x \times \id_\mathcal{X})^* \mathcal{O}_{\Delta_\mathcal{X}}$ and $\LL(\iota_x \times \id_{\mathcal{X}'})^* P$ are invariant under the diagonal and antidiagonal action, respectively;
  $\kappa^\alpha(x) \colonequals \R p_*\LL(\iota_x \times \id_\mathcal{X})^* \mathcal{O}_{\Delta_\mathcal{X}}$ is the skyscraper sheaf supported on $x$ with a $\G_m$-action of weight $-1$ and $\R p'_*\LL(\iota_x \times \id_{\mathcal{X}'})^* P = \Phi_P\bigl(\kappa^\alpha(x)\bigr)$ (using the kernel described in the previous sentence).
  Under these identifications, $\gamma_x$ becomes the natural map
  \[ \R\!\Hom_{\mathcal{X}_x \wedge \mathcal{X}}\bigl(\kappa^\alpha(x),\kappa^\alpha(x)\bigr) \to \R\!\Hom_{\mathcal{X}_x \wedge \mathcal{X}'}\bigl(\Phi_P\bigl(\kappa^\alpha(x)\bigr),\Phi_P\bigl(\kappa^\alpha(x)\bigr)\bigr) \]
  induced by functoriality.
  This is an isomorphism because $\Phi_P$ is fully faithful.
\end{proof}
More conceptually, Lemma~\ref{lem:Hochschild_cohom} is a consequence of the existence of twisted Hochschild cohomology.
We now prove a twisted analog of \cite[Lem.~5.2(i)]{MR3429474}.
\begin{lem}\label{lem:FM_morphism}
  The Fourier--Mukai kernel $P$ defines a morphism $\mu_P \colon \mathcal{X} \to s\Tw_{\mathcal{X}'/\C}$ such that for any $f \colon T \to \mathcal{X}$, the composition $\mu_P \circ f$ is given by $\LL(f \times \id_{\mathcal{X}'})^*P \in \Db(T \times \mathcal{X}')$.
\end{lem}
\begin{proof}
  We must verify that $P_T \colonequals \LL(f \times \id_{\mathcal{X}'})^*P$ meets the conditions of Definition~\ref{defn:moduli_complexes} for any $f \colon T \to \mathcal{X}$.
  As $P$ has weight $(1,1)$, the pullback $P_T$ to the $\G_m$-gerbe $T \times \mathcal{X}' \to T \times X'$ has weight $1$.
  By Lemma~\ref{lem:Hochschild_cohom}, $\Ext^i\bigl((P_T)_t,(P_T)_t\bigr) = 0$ for all $i < 0$ and all $t \in T$ closed.
  Hence, $P$ induces a morphism $\mathcal{X} \to \Tw_{\mathcal{X}'/\C}$.
  In fact, Lemma~\ref{lem:Hochschild_cohom} shows that $\G_{a,T} \xrightarrow{\sim} \mathcal{E}nd(P_T)$ for all $f \colon T \to \mathcal{X}$, so $\mathcal{X} \to \Tw_{\mathcal{X}'/\C}$ factors through a unique morphism
  \[ \mu_P \colon \mathcal{X} \to s\Tw_{\mathcal{X}'/\C}. \qedhere \]
\end{proof}
Since $\mathcal{X} \to X$ is the coarse moduli space, it is initial for morphisms to algebraic spaces.
Therefore, the composition $\mathcal{X} \xrightarrow{\mu_P} s\Tw_{\mathcal{X}'/\C} \to \sTwr_{\mathcal{X}'/\C}$ factors through a unique morphism $\bar{\mu}_P \colon X \to \sTwr_{\mathcal{X}'/\C}$.
\begin{defn}\label{defn:FM_fiber}
  Let $x \in X$ be a closed point.
  After choosing a lift $\tilde{x} \colon \Spec \C \to \mathcal{X}$, define
  \[ P_x \colonequals \LL(\tilde{x} \times \id_{\mathcal{X}'})^* P. \]
  As the pullbacks under any two such lifts are (non-canonically) isomorphic, $P_x$ is a well-defined element of $\sTwr_{\mathcal{X}'/\C}(\C)$.
\end{defn}
With this definition, $\bar{\mu}_P(x) = \bigl[\Phi_P\bigl(\kappa^\alpha(x)\bigr)\bigr] = P_x$, where, as before, $\kappa^\alpha(x)$ denotes the $\alpha$-twisted skyscraper sheaf at $x$.
\begin{lem}\label{lem:open_immers}
  The map $\bar{\mu}_P \colon X \to \sTwr_{\mathcal{X}'/\C}$ is an open immersion.
\end{lem}
\begin{proof}
  The proof carries over directly from \cite[Lem.~5.2(ii)]{MR3429474}, where it is given in complete detail.
  We use the full faithfulness of $\Phi_P$ to check that $\bar{\mu}_P$ is an \'etale monomorphism.
  For distinct closed points $x,y \in X$,
  \[ \Hom(P_x,P_y) \simeq \Hom\bigl(\kappa^\alpha(x),\kappa^\alpha(y)\bigr) = 0. \]
  Hence, $\bar{\mu}_P(x) = P_x$ and $\bar{\mu}_P(y) = P_y$ are not isomorphic and $\bar{\mu}_P$ is a monomorphism.

  Since $s\Tw_{\mathcal{X}'/\C} \to \sTwr_{\mathcal{X}'/\C}$ is a $\G_m$-gerbe, we have
  \[ T_{[P_x]}\bigl(\sTwr_{\mathcal{X}'/\C}\bigr) \simeq T_{[P_x]}\bigl(s\Tw_{\mathcal{X}'/\C}\bigr) \simeq \Ext^1(P_x,P_x) \]
  for any closed point $x \in X$ \cite[Cor.~3.17]{MR2493386}.
  The map on tangent spaces induced by $\bar{\mu}_P$ can be identified with the convolution with $P$
  \[ T_xX \simeq \Ext^1\bigl(\kappa^\alpha(x),\kappa^\alpha(x)\bigr) \to \Ext^1(P_x,P_x) \simeq T_{[P_x]}\bigl(\sTwr_{\mathcal{X}'/\C}\bigr) \]
  and is therefore an isomorphism.
  As $X$ is smooth, $\bar{\mu}_P$ is \'etale.
\end{proof}

\subsection{The Hilbert stack}\label{subsect:Hilbert-stack}

Quasi-finite morphisms to $\sTwr_{\mathcal{X}'/\C}$ such as $\bar{\mu}_P$ are classified by a Hilbert stack.
\begin{defn}
  Let $\mathcal{Y} \to S$ be a morphism of algebraic stacks.
  The Hilbert stack $\HS_{\mathcal{Y}/S}$ of $\mathcal{Y}/S$ is the $S$-stack whose fiber over a scheme $T \to S$ is given by
  \[ \HS_{\mathcal{Y}/S}(T) \colonequals
  \begin{Bmatrix}
    \text{quasi-finite, representable morphisms } Z \to \mathcal{Y} \times_S T \\
    \text{such that } Z \text{ is proper, flat, and finitely presented over } T
  \end{Bmatrix}. \]
\end{defn}
The following will later provide a good framework for deforming $\bar{\mu}_P$ and thus $P$.
\begin{lem}\label{lem:HS_algebraic}
  Let $S$ be a noetherian scheme, $X \to S$ a smooth and proper morphism of schemes, and $\mathcal{X}$ a $\G_m$-gerbe over $X$ corresponding to a torsion class.
  Let $Y \subseteq \sTwr_{\mathcal{X}/S}$ be an open subspace.
  Then $\HS_{Y/S}$ is represented by a locally finitely presented algebraic stack.
\end{lem}
\begin{proof}
  It suffices to verify the assumptions of \cite[Thm.~2]{MR3148551} that $Y$ is locally finitely presented, has affine stabilizers, and its diagonal morphism is quasi-compact and separated.
  As $Y$ is an algebraic space, the requirements on affine stabilizers and separated diagonal are automatic.
  The locally finite presentation follows from \cite[Cor.~3.22]{MR2493386}.
  Lastly, note that $Y \times_{\sTwr_{\mathcal{X}/S}} s\Tw_{\mathcal{X}/S} \subset \Tw_{\mathcal{X}/S}$ is an open substack.
  Hence, it has quasi-compact diagonal by Remark~\ref{rem:diag_qc}.
  Quasi-compactness of the diagonal of $Y$ now results from \cite[\href{http://stacks.math.columbia.edu/tag/0DQL}{0DQL}]{stacks-project}.
\end{proof}

\section{Deforming twisted Fourier--Mukai kernels}\label{sect:def}

\subsection{Deformations over \'etale neighborhoods}\label{subsect:deform_D}

We return to the setting of Theorem~\ref{thm:norhi}.
Fix two twisted K3 surfaces $(X,\alpha)$ and $(X',\alpha')$, and a twisted Fourier--Mukai equivalence $\Phi_P \colon \Db(X,\alpha) \xrightarrow{\sim} \Db(X',\alpha')$.
Given a deformation of $(X',\alpha')$, we will need to find corresponding deformations of $(X,\alpha)$ and $P$.
In view of \S\ref{subsection:FM_map}, the deformation theory of $(X,\alpha)$ and $P$ is closely linked to that of the moduli space of $\alpha'$-twisted perfect complexes on $X'$.
We begin with the latter.

Let $f \colon \mathfrak{X}' \to S$ be a projective family of K3 surfaces over a noetherian base scheme $S$.
Let $\mathcal{X}' \to \mathfrak{X}'$ be a $\G_m$-gerbe whose associated class in $\Hh^2(\mathfrak{X}',\G_m)$ is torsion.
Any twisted complex $E \in s\Tw_{\mathcal{X}'/S}(T)$ has an associated determinant $\det E$ \cite[Ch.~I]{MR0437541}, which is a line bundle on $\mathcal{X}' \times_S T$ of pure weight $\rk E$.
Mimicking \cite[Def.~2.2.6.26]{MR2309155}, we now define the open subspace of $\sTwr_{\mathcal{X}'/S}$ that corresponds to twisted perfect complexes with trivial determinant.

Let $s\Tw_{\mathcal{X}'/S}(0) \subset s\Tw_{\mathcal{X}'/S}$ be the open substack of complexes of rank $0$ and $\sTwr_{\mathcal{X}'/S}(0)$ its coarse moduli space.
Since the determinant of any object $E$ of $s\Tw_{\mathcal{X}'/S}(0)(T)$ is of weight $0$, it is the pullback of a line bundle on $\mathfrak{X}' \times_S T$.
This yields a morphism $\det \colon s\Tw_{\mathcal{X}'/S}(0) \to \Pics_{\mathfrak{X}'/S}$ to the Picard stack of $\mathfrak{X}'$.
The induced morphism on coarse spaces is, by slight abuse of notation, again denoted $\det \colon \sTwr_{\mathcal{X}'/S}(0) \to \Pic_{\mathfrak{X}'/S}$.

The algebraic space $\Pic_{\mathfrak{X}'/S}$ is unramified over $S$, essentially because $\Hh^1(\mathfrak{X}'_s,\mathcal{O}_{\mathfrak{X}'_s})$ is trivial for all $s \in S$ (see e.g.\ \cite[Prop.~3.1.3]{MR2263236}).
Thus, the section $\mathcal{O}_{\mathfrak{X}'} \colon S \to \Pic_{\mathfrak{X}'/S}$ corresponding to the structure sheaf is an open immersion.
Define
\[ \sTwr^0_{\mathcal{X}'/S} \colonequals \sTwr_{\mathcal{X}'/S}(0) \times_{\det,\Pic_{\mathfrak{X}'/S},\mathcal{O}_{\mathfrak{X}'}} S. \]
This is an open subspace of $\sTwr_{\mathcal{X}'/S}$;
its preimage $s\Tw^0_{\mathcal{X}'/S}$ in the stack $s\Tw_{\mathcal{X}'/S}$ parameterizes all twisted perfect complexes with trivial determinant because every such complex has rank $0$.

In Corollary~\ref{cor:trivial_det}, we will reduce Theorem~\ref{thm:norhi} to kernels with fiberwise trivial determinant.
Proposition~\ref{prop:Twr_smooth} below shows that this simplifies the involved deformation theory significantly.
A natural framework for its proof is provided by derived algebraic geometry; cf.\ \cite[\S 5.1]{MR3341464}.
To keep the discussion simple, we isolate the input from derived methods in the following lemma, which could presumably also be verified directly with a \v{C}ech cocycle computation.
\begin{lem}\label{lem:obstruction}
  The algebraic stack $s\Tw_{\mathcal{X}'/S}$ has an obstruction theory with the following properties for any small extension $A' \to A$ of local Artinian $\C$-algebras with residue field $\C$, any morphism $\Spec A' \to S$, and any object $E$ of $s\Tw_{\mathcal{X}'/S}(\Spec A)$:
  \begin{enumerate}[label={\upshape(\alph*)}]
    \item\label{lem:obstruction_module} The obstruction module for the deformation situation above is $\Ext^2\bigl(E_0,E_0 \otimes^\LL_\C I\bigr)$, where $E_0 \colonequals E \otimes^\LL_A \C$ is the restriction of $E$ to the special fiber $\mathfrak{X}'_0$ of $\mathfrak{X}'_A$ and $I = \ker(A' \to A)$.
    \item\label{lem:obstruction_class} When $E$ is an object of $s\Tw_{\mathcal{X}'/S}(0)(\Spec A)$, the trace morphism $\tr \colon \Ext^2\bigl(E_0,E_0 \otimes^\LL_\C I\bigr) \to \Ext^2\bigl(\mathcal{O}_{\mathfrak{X}'_0},\mathcal{O}_{\mathfrak{X}'_0} \otimes^\LL_\C I\bigr) \simeq \Hh^2(\mathfrak{X}'_0,f^*I)$ maps the obstruction class $o(E)$ to $o(\det E)$.
  \end{enumerate}
\end{lem}
\begin{proof}
  Fix a $\mu_n$-gerbe $\mathcal{Y}' \to \mathfrak{X}'$ whose cohomology class maps to the class of $\mathcal{X}'$ under the natural morphism $\Hh^2(\mathfrak{X}',\mu_n) \to \Hh^2(\mathfrak{X}',\G_m)$ induced by the inclusion.
  It suffices to prove the statement for $s\Tw_{\mathcal{Y}'/S}$ instead of $s\Tw_{\mathcal{X}'/S}$; cf.\ Remark~\ref{rem:G_m_mu_n}.
  For ease of notation, set $s\Tw \colonequals s\Tw_{\mathcal{Y}'/S}$.

  Since $s\Tw$ is an algebraic stack locally of finite type over $S$, its cotangent complex $\LL_{s\Tw/S}$ endows it with a canonical obstruction theory \cite[Rmk.~1.7]{MR2206635};
  for any deformation situation as above with corresponding morphism $\nu_E \colon \Spec A \to s\Tw$, the obstruction module is $\Ext^1_A\bigl(\LL\nu^*_E \LL_{s\Tw/S},I\bigr)$.
  On the other hand, in their proof of Theorem~\ref{thm:tw_mod_ex}, \cite{MR2493386,MR2957304} realize $s\Tw$ as the truncation of a locally finitely presented derived Artin stack $s\Tw^{\der}$.
  The closed immersion $j \colon s\Tw \hookrightarrow s\Tw^{\der}$ induces a morphism $\LL j^*\LL_{s\Tw^{\der}/S} \to \LL_{s\Tw/S}$ and thus by \cite[Prop.~1.2]{MR3341464} an inclusion $\Ext^1_A\bigl(\LL\nu^*_E \LL_{s\Tw/S},I\bigr) \hookrightarrow \Ext^1_A\bigl(\LL\nu^*_E\LL j^* \LL_{s\Tw^{\der}/S},I\bigr)$.
  Since the latter is naturally isomorphic to
  \[ \Ext^1_\C\bigl(\LL\nu^*_{E_0}\LL j^* \LL_{s\Tw^{\der}/S},I\bigr) \simeq \Ext^1_\C\bigl(\R\!\Hom(E_0,E_0)^\vee[-1],I\bigr) \simeq \Ext^2\bigl(E_0,E_0 \otimes^\LL_\C I\bigr), \]
  by the adjunction between scalar extension and scalar restriction and \cite[Cor.~3.17]{MR2493386}, this defines an obstruction theory on $s\Tw$ satisfying \ref{lem:obstruction_module}.

  The construction in \cite[\S 3.1]{MR3341464} yields a determinant morphism $s\Tw^{\der} \to \Pics^{\der}_{\mathcal{Y}'/S}$ of derived stacks over $S$ (under the standing assumption that we work over the ground field $\C$).
  Via the correspondence between open derived substacks of a derived stack and open substacks of its truncation \cite[Prop.~2.1]{MR3341464}, it restricts to a morphism $\det^{\der} \colon s\Tw^{\der}(0) \to \Pics^{\der}_{\mathfrak{X}'/S}$, whose truncation is the morphism $\det$ defined earlier.
  For every deformation situation as before, we obtain the commutative diagram
  \[ \begin{tikzcd}
      \Ext^1_A\bigl(\LL\nu^*_E \LL_{s\Tw(0)/S},I\bigr) \arrow[r,hook] \arrow[d] & \Ext^1_A\bigl(\LL\nu^*_E\LL j^* \LL_{s\Tw^{\der}(0)/S},I\bigr) \arrow[r,"\sim"] \arrow[d] & \Ext^2\bigl(E_0,E_0 \otimes^\LL_\C I\bigr) \arrow[d] \\
      \Ext^1_A\bigl(\LL\nu^*_{\det E} \LL_{\Pics_{\mathfrak{X}'/S}/S},I\bigr) \arrow[r,hook] & \Ext^1_A\bigl(\LL\nu^*_{\det E}\LL j^* \LL_{\Pics^{\der}_{\mathfrak{X}'/S}/S},I\bigr) \arrow[r,"\sim"] & \Ext^2\bigl(\mathcal{O}_{\mathfrak{X}'_0},\mathcal{O}_{\mathfrak{X}'_0} \otimes^\LL_\C I\bigr).
  \end{tikzcd} \]
  By the functoriality of Olsson's obstruction theory, the left arrow maps $o(E)$ to $o(\det E)$.
  On the other hand, the tangent map $\R\!\Hom(E_0,E_0)[1] \to \R\!\Hom\bigl(\mathcal{O}_{\mathfrak{X}'_0},\mathcal{O}_{\mathfrak{X}'_0}\bigr)[1]$ of $\det^{\der}$ induces the right arrow.
  The same proof as for \cite[Prop.~3.2]{MR3341464} shows that it is given by $\tr_{E_0}[1]$ if we note that every twisted perfect complex on the special fiber $\mathcal{Y}'_0$ is strictly perfect \cite[Cor.~2.2.7.21]{MR2309155}.
  Property \ref{lem:obstruction_class} follows.
\end{proof}
\begin{prop}\label{prop:Twr_smooth}
  The moduli space $\sTwr^0_{\mathcal{X}'/S}$ is smooth over $S$.
\end{prop}
\begin{proof}
  By \cite[\href{http://stacks.math.columbia.edu/tag/0APP}{Tag~0APP}]{stacks-project}, it suffices to test the lifting criterion for formal smoothness for any small extension $A' \to A$ of local Artinian $\C$-algebras with residue field $\C$, any morphism $\Spec A' \to S$, and any $E \in \sTwr^0_{\mathcal{X}'/S}(\Spec A)$.
  Let as before $E_0 \colonequals E \otimes^\LL_A \C$ be the restriction of $E$ to the special fiber $\mathfrak{X}'_0$ of $\mathfrak{X}'_A$, and $I \colonequals \ker(A' \to A)$.
  Using the obstruction theory from Lemma~\ref{lem:obstruction}, we must show that the obstruction class $o(E) \in \Ext^2_A(E_0,E_0 \otimes^\LL_\C I)$ vanishes.

  Consider the trace morphism $\tr \colon \R\mathcal{H}om(E_0,E_0) \to \mathcal{O}_{\mathfrak{X}'_0}$.
  Grothendieck duality for the untwisted complex $\R\mathcal{H}om(E_0,E_0)$ yields the following commutative diagram:
  \[ \begin{tikzcd}
      \Ext^2\bigl(E_0,E_0 \otimes^\LL_\C I\bigr) \arrow[r, "\sim"]\arrow[d,"\tr"] & \Hom(\Hom(E_0,E_0),I) \\
      \Hh^2(\mathfrak{X}'_0,f^*I) \arrow[r, "\sim"] & \Hom(\Hh^0(\mathfrak{X}'_0,\mathcal{O}_{\mathfrak{X}'_0}),I). \arrow[u,"\tr"]
  \end{tikzcd} \]
  As $E_0$ is simple, the right arrow, and thereby also the left arrow, is an isomorphism.
  Since $\det E \simeq \mathcal{O}_{\mathfrak{X}'_A}$ deforms over $A'$, the statement now follows from Property~\ref{lem:obstruction_class} in Lemma~\ref{lem:obstruction}.
\end{proof}
\begin{prop}\label{prop:deformation}
  Let $(X,\alpha)$, $(X',\alpha')$, and $P \in \Db (X \times X', \alpha^{-1} \boxtimes \alpha')$ as before.
  Assume that $\rk P_x = 0$ and $\det P_x \simeq \mathcal{O}_{X'}$ for all $x \in X$ closed.
  Let $S$ be a smooth and quasi-compact scheme over $\C$.
  Let $\mathfrak{X}' \to S$ be a projective family of K3 surfaces and $\mathcal{X}'$ a $\G_m$-gerbe over $\mathfrak{X}'$.
  Assume there exists $s \in S$ with fiber $\mathfrak{X}'_s = X'$ such that the class in $\Hh^2(X',\G_m)$ associated to $\mathcal{X}'_s$ is $\alpha'$.

  Then there exists a connected \'etale neighborhood $(U,u) \to (S,s)$, together with
  \begin{enumerate}[label={\upshape(\roman*)}]
    \item\label{item:deformation_family} a family of K3 surfaces $\mathfrak{X}_U \to U$ with fiber $(\mathfrak{X}_U)_u = X$,
    \item\label{item:deformation_gerbe} a $\G_m$-gerbe $\mathcal{X}_U$ over $\mathfrak{X}_U$ whose associated class $\alpha^{-1}_U \in \Hh^2(\mathfrak{X}_U,\G_m)$ is torsion and pulls back to $\alpha^{-1}$ on $(\mathfrak{X}_U)_u = X$, and
    \item\label{item:deformation_kernel} a kernel $\mathfrak{P} \in \Db(\mathfrak{X}_U \times \mathfrak{X}'_U,\alpha^{-1}_U \boxtimes \alpha'_U)$ (where $\mathfrak{X}'_U \colonequals \mathfrak{X}' \times_S U$ and $\alpha'_U \in \Hh^2(\mathfrak{X}'_U,\G_m)$ is the cohomology class associated to $\mathcal{X}' \times_S U$) with $\mathfrak{P}_u = P$ that induces a Fourier--Mukai equivalence on each fiber.
  \end{enumerate}
\end{prop}
\begin{proof}
  \ref{item:deformation_family}.
  Let $A \colonequals \mathcal{O}_{S,s}$, $\mathfrak{m} \subset A$ its maximal ideal, $A_n \colonequals A/\mathfrak{m}^{n+1}$, and $\widehat{A}$ the $\mathfrak{m}$-adic completion of $A$.
  The fiber $\mathcal{X}'_s$ represents $\alpha' \in \Hh^2(X',\G_m)$.
  Choose a $\G_m$-gerbe $p \colon \mathcal{X}_s \to X$ that represents $\alpha^{-1}$.
  To lighten notation, define $s\Tw_V \colonequals s\Tw_{\mathcal{X}'_V / V}$ for any $V \to S$, and similarly $\sTwr_V$ and $\sTwr^0_V$.
  Since $\mathfrak{X}'$ is regular and quasi-compact, $\Hh^2(\mathfrak{X}',\G_m)$ is torsion by a theorem of Grothendieck.
  Thus, Theorem~\ref{thm:tw_mod_ex} applies.

  As $\det P_x \simeq \mathcal{O}_{X'}$ for all $x \in X$ closed, the open immersion $\bar{\mu}_P$ from Lemma~\ref{lem:open_immers} factors through some $\bar{\mu}_0 \colon X \hookrightarrow \sTwr^0_{A_0}$ by the seesaw theorem.
  Moreover, for each $n \in \Z_{>0}$, the inclusion $\sTwr^0_{A_0} \subset \sTwr^0_{A_n}$ is given by the pullback of the nilpotent closed immersion $\Spec A_0 \hookrightarrow \Spec A_n$ and is therefore a thickening.
  Hence, we can extend $\bar{\mu}_0$ to a compatible system of open immersions $\bar{\mu}_n \colon X_n \hookrightarrow \sTwr^0_{A_n}$; cf.\ \cite[\href{http://stacks.math.columbia.edu/tag/05ZP}{Tag~05ZP}]{stacks-project}.
  Since $\sTwr^0_{A_n}$ is flat over $\Spec A_n$ by Proposition~\ref{prop:Twr_smooth}, so too is $X_n$.
  Furthermore, $X_n$ is proper over $\Spec A_n$ because properness is preserved under infinitesimal deformations of morphisms that are locally of finite type \cite[\href{http://stacks.math.columbia.edu/tag/0BPJ}{Tag~0BPJ}]{stacks-project}.
  Consequently, we obtain a formal object of the Hilbert stack
  \[ (\bar{\mu}_n) \in \lim_n \HS_{\sTwr^0_S/S}(\Spec A_n). \]

  Since $\HS_{\sTwr^0_S/S}$ is a locally finitely presented algebraic stack by Lemma~\ref{lem:HS_algebraic}, the effectivity in Artin's axioms guarantees the existence of a $\widehat{\bar{\mu}} \in \HS_{\sTwr^0_S/S}\bigl(\Spec \widehat{A}\bigr)$ restricting to $(\bar{\mu}_n)$.
  We then use Artin approximation to find an \'etale neighborhood $(U,u) \to (S,s)$ and $\bar{\mu}_U \in \HS_{\sTwr^0_S/S}(U)$ such that $\bar{\mu}_U\big|_{\Spec A_0} \simeq \widehat{\bar{\mu}}\big|_{\Spec A_0} = \bar{\mu}_0$.
  Put differently, there is a proper, flat, and finitely presented algebraic stack $\mathfrak{X}_U$ over $U$, and a quasi-finite, representable morphism
  \[ \begin{tikzcd}
      \mathfrak{X}_U \arrow[rr, "\bar{\mu}_U"] \arrow[rd, "f"'] && \sTwr^0_U \mathrlap{= \sTwr^0_S \times_S U} \arrow[ld] \\
      & U &
  \end{tikzcd} \]
  which restricts to $\bar{\mu}_0$ over $u$.

  As $\bar{\mu}_U$ is representable and $\sTwr^0_U$ is an algebraic space, so too is $\mathfrak{X}_U$.
  Since the fiber of $\mathfrak{X}_U$ over $u$ is a K3 surface, $f\colon \mathfrak{X}_U \to U$ becomes, upon replacing $U$ by an open neighborhood of $u$, a family of K3 surfaces:
  the proof of this fact for the universal family over the Hilbert scheme given in \cite[Prop.~5.2.1]{huybrechts2016lectures} applies unchanged to any flat and proper morphism of locally noetherian algebraic spaces.
  Finally, we can pass to a further open neighborhood of $u$ to ensure that $U$ is quasi-compact and connected.

  \ref{item:deformation_gerbe}.
  Define
  \[ \mathcal{X}_U \colonequals \mathfrak{X}_U \times_{\sTwr_U} s\Tw_U. \]
  As $s\Tw_U \to \sTwr_U$ is a $\G_m$-gerbe, so too is $\mathcal{X}_U \to \mathfrak{X}_U$.
  Let $\alpha^{-1}_U \in \Hh^2(\mathfrak{X}_U,\G_m)$ be its associated cohomology class.
  By the algebraic space version of Grothendieck's theorem \cite[Cor.~3.1.3.4]{MR2388554}, $\Hh^2(\mathfrak{X}_U,\G_m)$, and hence $\alpha^{-1}_U$, is torsion.
  Since $P$ is a complex on $\mathcal{X}_s \times \mathcal{X}'_s$ of weight $(1,1)$, the map $(p, \mu_P) \colon \mathcal{X}_s \to X \times_{\sTwr_\C} s\Tw_\C = (\mathcal{X}_U)_u$ is a morphism of $\G_m$-gerbes over $X$, thus an equivalence.
  As the identification of isomorphism classes of $\G_m$-gerbes with second cohomology classes of $\G_m$ is functorial in the base space \cite[Cor.~V.1.5.3]{MR0344253}, $\alpha_U^{-1}$ pulls back to $\alpha^{-1}$ on $(\mathfrak{X}_U)_u$.

  \ref{item:deformation_kernel}.
  The projection $\mathcal{X}_U \to s\Tw_U$ corresponds to a simple, perfect complex $\mathfrak{P}$ on $\mathcal{X}_U \times \mathcal{X}'_U$.
  The $\G_m$-gerbe structure on $s\Tw_U$ is such that the inertial action of $\mathcal{X}_U$ on $\mathfrak{P}$ comes from the natural map $\G_m \to \mathcal{A}ut(\mathfrak{P})$.
  Thus, $\mathfrak{P}$ is of weight $1$ with respect to $\mathcal{X}_U$.
  As $\mathfrak{P}$ is, by definition, also of weight $1$ with respect to $\mathcal{X}'_U$, we conclude $\mathfrak{P} \in \Db(\mathfrak{X}_U \times \mathfrak{X}'_U,\alpha^{-1}_U \boxtimes \alpha'_U)$.

  Let $\pi_{\mathfrak{X}_U}$ and $\pi_{\mathfrak{X}'_U}$ denote the projections from $\mathfrak{X}_U \times \mathfrak{X}'_U$ onto $\mathfrak{X}_U$ and $\mathfrak{X}'_U$, respectively.
  Consider the complexes
  \[ \mathfrak{Q} \colonequals \mathfrak{P}^\vee \otimes^\LL \pi^*_{\mathfrak{X}'_U}\omega_{\mathfrak{X}'_U/S}[2] \quad \text{and} \quad \mathfrak{R} \colonequals \mathfrak{P}^\vee \otimes^\LL \pi^*_{\mathfrak{X}_U}\omega_{\mathfrak{X}_U/S}[2] \]
  in $\Db(\mathfrak{X}'_U \times \mathfrak{X}_U,(\alpha'_U)^{-1} \boxtimes \alpha_U)$.
  By Grothendieck duality, the associated Fourier--Mukai transforms $\Phi_{\mathfrak{Q}_t} \colon \Db\bigl((\mathfrak{X}'_U)_t\bigr) \to \Db\bigl((\mathfrak{X}_U)_t\bigr)$ and $\Phi_{\mathfrak{R}_t} \colon \Db\bigl((\mathfrak{X}'_U)_t\bigr) \to \Db\bigl((\mathfrak{X}_U)_t\bigr)$ are left-adjoint and right-adjoint to $\Phi_{\mathfrak{P}_t} \colon \Db\bigl((\mathfrak{X}_U)_t\bigr) \to \Db\bigl((\mathfrak{X}'_U)_t\bigr)$, respectively, for all closed $t\in U$.
  In the discussion preceding \cite[Prop.~3.3]{MR3429474}, Lieblich and Olsson explicitly describe the corresponding adjunction maps.
  All steps in their constructions can be carried out for twisted sheaves and base scheme $U$ and commute with base change to fibers over closed points of $U$.
  We obtain morphisms of complexes
  \begin{IEEEeqnarray*}{lr}
    \eta \colon \R\pi_{\mathfrak{X}_U \times \mathfrak{X}_U,*}\Bigl(\LL\pi^*_{\mathfrak{X}_U \times \mathfrak{X}'_U}\mathfrak{P} \otimes^\LL \LL\pi^*_{\mathfrak{X}'_U \times \mathfrak{X}_U}\mathfrak{Q}\Bigr) \to \R\Delta_*\mathcal{O}_{\mathfrak{X}_U} & \quad \text{and} \\
    \varepsilon \colon \R\Delta_*\mathcal{O}_{\mathfrak{X}_U} \to \R\pi_{\mathfrak{X}_U \times \mathfrak{X}_U,*}\Bigl(\LL\pi^*_{\mathfrak{X}_U \times \mathfrak{X}'_U}\mathfrak{P} \otimes^\LL \LL\pi^*_{\mathfrak{X}'_U \times \mathfrak{X}_U}\mathfrak{R}\Bigr) &
  \end{IEEEeqnarray*}
  whose fibers over all closed $t \in U$ determine counit and unit of the above adjunctions.
  The same argument as in \cite[Prop.~3.3]{MR3429474} then shows that $\Phi_{\mathfrak{P}_t}$ is an equivalence if and only if $\eta_t$ and $\varepsilon_t$ are isomorphisms.
  After replacing $U$ with an open subscheme, we may therefore assume that $\mathfrak{P}$ induces fiberwise equivalences.
\end{proof}

\subsection{Action on cohomology}\label{subsect:signed_deformation}

We continue to work in the setup of Proposition~\ref{prop:deformation}.
Via GAGA for algebraic spaces, we associate to $U$ a connected complex manifold $U^{\an}$ and to $\mathfrak{X}_U$ and $\mathfrak{X}'_U$ smooth, proper families of complex K3 surfaces $f^{\an} \colon \mathfrak{X}^{\an}_U \to U^{\an}$ and $f^{\prime\an} \colon \mathfrak{X}^{\prime \an}_U \to U^{\an}$.
We further associate analytic torsion classes $\alpha^{\an}_U \in \Hh^2\bigl(\mathfrak{X}^{\an}_U,\mathcal{O}^*_{\mathfrak{X}^{\an}_U}\bigr)$ and $\alpha^{\prime \an}_U \in \Hh^2\bigl(\mathfrak{X}^{\prime\an}_U,\mathcal{O}^*_{\mathfrak{X}^{\prime\an}_U}\bigr)$ to $\alpha_U$ and $\alpha'_U$.
Analytification for separated Deligne--Mumford stacks locally of finite type over $\C$ \cite[Ch.~5]{toen-thesis} produces a perfect complex $\mathfrak{P}^{\an} \in \Db\bigl(\mathfrak{X}^{\an}_U \times_{U^{\an}} \mathfrak{X}^{\prime\an}_U,(\alpha^{\an}_U)^{-1} \boxtimes \alpha^{\prime\an}_U\bigr)$.
Following common misuse of notation, we will from here on omit the superscript ``$\an$.''

We now explain how the action of the Fourier--Mukai kernel $\mathfrak{P}$ on cohomology still exists in this relative, complex setting.
First, as $\Rr^3f_*\Z = 0$ and $\alpha_U$ is torsion, we can choose $\mathfrak{B} \in \Hh^0(U,\Rr^2f_*\Q)$ such that the image of $\mathfrak{B}$ under the fiberwise exponential map $\Hh^0(U,\Rr^2f_*\Q) \to \Hh^0\bigl(U,\Rr^2f_*\mathcal{O}^*_{\mathfrak{X}_U}\bigr)$ coincides with the image of $\alpha_U$ under the edge morphism $\Hh^2\bigl(\mathfrak{X}_U,\mathcal{O}^*_{\mathfrak{X}_U}\bigr) \to \Hh^0\bigl(U,\Rr^2f_*\mathcal{O}^*_{\mathfrak{X}_U}\bigr)$.
We choose $\mathfrak{B}' \in \Hh^0(U,\Rr^2f'_*\Q)$ in the same way.

Let $g \colon \mathfrak{X}_U \times \mathfrak{X}'_U \to U$ be the structure morphism.
Next, we define a twisted Chern character $\ch^{-\mathfrak{B} \boxplus \mathfrak{B}'}_U(\mathfrak{P}) \in \oplus_i \Hh^0(U,\Rr^ig_*\Q)$, which agrees with the one from \cite[Prop.~1.2]{MR2179782} on each fiber.
Let $t \in U$.
Pick an analytic neighborhood $V$ of $t$ such that $(-\mathfrak{B} \boxplus \mathfrak{B}')|_V$ is represented by a class $\beta_V \in \Hh^2(\mathfrak{X}_V \times \mathfrak{X}'_V,\Q)$.
After shrinking $V$ if necessary, we may assume that $\exp(\beta_V) = \alpha^{-1}_V \boxtimes \alpha'_V$.
As in \cite[Prop.~1.2]{MR2179782}, $\beta_V$ determines  an $\alpha^{-1}_V \boxtimes \alpha'_V$-twisted, differentiable line bundle $L$ on $\mathfrak{X}_V \times \mathfrak{X}'_V$ because the sheaf $\mathcal{C}^\infty$ is fine and therefore acyclic.

By \cite[Exp.~I, Cor.~4.19.1]{SGA6}, $\mathfrak{P} \otimes L^{-1}$ is a perfect complex on the $\mathcal{C}^\infty$-topos of $\mathfrak{X}_V \times \mathfrak{X}'_V$.
After we replace $V$ by a relatively compact, open neighborhood of $t$ if necessary, \cite[Exp.~II, Prop.~2.3.2.c)]{SGA6} shows that $\mathfrak{P} \otimes L^{-1}$ admits a global resolution by a bounded complex of locally free sheaves.
In particular, $\ch(\mathfrak{P} \otimes L^{-1})$ is well-defined.
Let $\ch^{-\mathfrak{B} \boxplus \mathfrak{B}'}_V(\mathfrak{P})$ be its image under $\oplus_i \Hh^i(\mathfrak{X}_V \times \mathfrak{X}'_V, \Q) \to \oplus_i \Hh^0(V,\Rr^ig_*\Q)$.
By construction, the germ at each point of $V$ is the twisted Chern character from \cite[Prop.~1.2]{MR2179782} on that fiber.
Thus, the various $\ch^{-\mathfrak{B} \boxplus \mathfrak{B}'}_V(\mathfrak{P})$ for open $V \subseteq U$ only depend on $-\mathfrak{B} \boxplus \mathfrak{B}'$ and glue to a global section $\ch^{-\mathfrak{B} \boxplus \mathfrak{B}'}_U(\mathfrak{P}) \in \oplus_i \Hh^0(U,\Rr^ig_*\Q)$.

Let $T_{(\mathfrak{X}_U \times \mathfrak{X}'_U)/U}$ be the relative holomorphic tangent bundle.
Define the relative Todd class $\td_U(\mathfrak{X}_U \times \mathfrak{X}'_U)$ as the image of $\td\bigl(T_{(\mathfrak{X}_U \times \mathfrak{X}'_U)/U}\bigr)$ under the edge morphism $\oplus_i \Hh^i(\mathfrak{X}_U \times \mathfrak{X}'_U,\Q) \to \oplus_i\Hh^0(U,\Rr^ig_*\Q)$.
Then the \emph{relative twisted Mukai vector} of $\mathfrak{P}$ is
\[ v^{-\mathfrak{B}\boxplus\mathfrak{B}'}_U(\mathfrak{P}) \colonequals \ch^{-\mathfrak{B} \boxplus \mathfrak{B}'}_U(\mathfrak{P}) \cdot \sqrt{\td_U(\mathfrak{X}_U \times \mathfrak{X}'_U)} \in \bigoplus_i\Hh^0(U,\Rr^ig_*\Q). \]
In fact, since $\mathfrak{X}_U \to U$ and $\mathfrak{X}'_U \to U$ are families of K3 surfaces, $v^{-\mathfrak{B}\boxplus\mathfrak{B}'}_U(\mathfrak{P}) \in \bigoplus_i\Hh^0(U,\Rr^ig_*\Z)$;
this can be checked on fibers, where it follows from the discussion preceding \cite[Prop.~4.3]{MR2179782}.

Pullback of sections induces a unique map of $\delta$-functors
\[ \pi^*_{\mathfrak{X}_U} \colon \Rr^if_*\Z \to \Rr^ig_*\pi^{-1}_{\mathfrak{X}_U}\Z \simeq \Rr^ig_*\Z. \]
Moreover, by Verdier duality, pullback along $\pi_{\mathfrak{X}'_U}$ admits a dual pushforward map
\[ \pi_{\mathfrak{X}'_U,*} \colon \Rr^ig_*\Z \to \Rr^{i-4}f'_*\Z. \]
The cohomological twisted Fourier--Mukai transform $\Phi^{\Hh}_\mathfrak{P}$ associated to $\mathfrak{P}$ is given by
\[\begin{tikzcd}[column sep=huge]
  & \oplus_i \Rr^ig_*\Z \arrow[r,"\cup v^{-\mathfrak{B} \boxplus \mathfrak{B}'}_U(\mathfrak{P})"] & \oplus_i \Rr^ig_*\Z \arrow[dr,"\pi_{\mathfrak{X}'_U,*}"] & \\
  \oplus_i \Rr^if_*\Z \arrow[ur,"\pi^*_{\mathfrak{X}_U}"] \arrow[rrr,"\Phi^{\Hh}_\mathfrak{P}"] &&& \oplus_i \Rr^if'_*\Z.
\end{tikzcd}\]
It reduces to the usual twisted Fourier--Mukai transform for $\mathfrak{P}_t$ over each $t\in U$.

\begin{prop}\label{prop:signed_deform}
  The fiberwise cohomological realizations
  \[ \Phi^{\Hh}_{\mathfrak{P}_t} \colon \widetilde{\Hh}\bigl((\mathfrak{X}_U)_t,\mathfrak{B}_t,\Z\bigr) \xrightarrow{\sim} \widetilde{\Hh}\bigl((\mathfrak{X}'_U)_t,\mathfrak{B}'_t,\Z\bigr) \]
  are either signed for all $t \in U$ or non-signed for all $t \in U$.
\end{prop}
\begin{proof}
  Let $\widetilde{\mathcal{H}} \colonequals \oplus_i\Rr^if_*\R \otimes_\R \mathcal{C}^\infty_U$ be the relative cohomology bundle for $\mathfrak{X}_U \to U$ and $\Gr^+_4\bigl(\widetilde{\mathcal{H}}\bigr) \to U$ its associated Grassmann bundle of positive-definite, oriented, 4-dimensional subspaces.
  Let $\widetilde{\mathcal{H}}'$ and $\Gr^+_4\bigl(\widetilde{\mathcal{H}}'\bigr) \to U$ be the corresponding bundles for $\mathfrak{X}'_U \to U$.
  As $\Phi^{\Hh}_\mathfrak{P}$ is compatible with the relative Mukai pairings $\bigl(\oplus_i\Rr^if^{(\prime)}_*\R\bigr) \otimes \bigl(\oplus_i\Rr^if^{(\prime)}_*\R\bigr) \to \Rr^4f^{(\prime)}_*\R \simeq \R$, it induces an isomorphism of bundles
  \[ \Phi^{\Gr}_\mathfrak{P} \colon \Gr^+_4\bigl(\widetilde{\mathcal{H}}\bigr) \xrightarrow{\sim} \Gr^+_4\bigl(\widetilde{\mathcal{H}}'\bigr). \]

  Choose a base point $u \in U$.
  The long exact sequence of homotopy groups for the fibration $\Gr^+_4\bigl(\widetilde{\mathcal{H}}\bigr) \to U$ ends in
  \[ \pi_1(U) \to \pi_0\Bigl(\Gr^+_4\bigl(\widetilde{\Hh}\bigl((\mathfrak{X}_U)_u,\R\bigr)\bigr)\Bigr) \to \pi_0\bigl(\Gr^+_4(\widetilde{\mathcal{H}})\bigr) \to \pi_0(U). \]
  The first map is induced by the monodromy representation $\rho \colon \pi_1(U) \to \Oo\bigl(\widetilde{\Hh}\bigl((\mathfrak{X}_U)_u,\R\bigr)\bigr)$.
  As, in the notation of Definition~\ref{defn:or_pres}, $\im(\rho) \subseteq \Oo^+\bigl(\widetilde{\Hh}\bigl((\mathfrak{X}_U)_u,\R\bigr)\bigr)$ (see e.g.\ \cite[Prop.~7.5.5]{huybrechts2016lectures}), $\pi_1(U) \to \pi_0\bigl(\Gr^+_4\bigl(\widetilde{\Hh}\bigl((\mathfrak{X}_U)_u,\R\bigr)\bigr)\bigr)$ is trivial.
  Since $U$ is connected, $\Gr^+_4(\widetilde{\mathcal{H}})$ must have two connected components, which correspond to the two connected components of $\Gr^+_4\bigl(\widetilde{\Hh}\bigl((\mathfrak{X}_U)_u,\R\bigr)\bigr)$.
  Any K\"ahler class on a fiber $(\mathfrak{X}_U)_t$ deforms differentiably over a small neighborhood of $t$ \cite[Thm.~15]{MR0115189} and both the holomorphic forms $\sigma_t$ and the $B$-fields $\mathfrak{B}_t$ vary differentiably with $t$.
  Thus, there is one connected component of $\Gr^+_4\bigl(\widetilde{\mathcal{H}}\bigr)$ which in every fiber picks out the connected component of $\Gr^+_4\bigl(\widetilde{\Hh}\bigl((\mathfrak{X}_U)_t,\mathfrak{B}_t,\R\bigr)\bigr)$ defining the natural positive sign structure.
  The same holds for $\mathfrak{X}'_U$.
  Hence, $\Phi^{\Hh}_{\mathfrak{P}_t}$ is either signed for all $t \in U$ or non-signed for all $t \in U$, depending on whether $\Phi^{\Gr}_\mathfrak{P}$ maps the connected components of $\Gr^+_4\bigl(\widetilde{\mathcal{H}}\bigr)$ and $\Gr^+_4\bigl(\widetilde{\mathcal{H}}'\bigr)$ that determine the respective positive sign structures into each other or not.
\end{proof}

\section{Proof of Theorem~\ref{thm:norhi}}\label{sect:pf}

We conclude with the proof of Theorem~\ref{thm:norhi}, which establishes the strong form of the derived global Torelli theorem for twisted K3 surfaces (Theorem~\ref{thm:stdgt}).
\norhi*
The strategy of the proof is explained in \S\ref{sect:intro}.
Note that whether or not $\Phi^{\Hh}_P$ is signed does not depend on the choice of $B$-field lifts.

\subsection{Initial reduction}\label{subsect:init_red}

In order to employ the deformation theory of \S\ref{sect:def}, we must first use \cite[Thm.~0.1]{MR2310257} to reduce to the case when $\rk P_x = 0$ and $\det P_x \simeq \mathcal{O}_{X'}$ for all $x \in X$ closed.
\begin{lem}\label{lem:reduction-id}
  Let $(X,\alpha)$ and $(X',\alpha')$ be two complex, projective twisted K3 surfaces with $B$-field lifts $B \in \Hh^2(X,\Q)$ and $B' \in \Hh^2(X',\Q)$.
  Let $\Phi_P \colon \Db(X,\alpha) \xrightarrow{\sim} \Db(X',\alpha')$ be a twisted Fourier--Mukai equivalence.
  If the induced Hodge isometry $\Phi^{\Hh}_P \colon \widetilde{\Hh}(X,B,\Z) \xrightarrow{\sim} \widetilde{\Hh}(X',B',\Z)$ is non-signed, then there is another Fourier--Mukai equivalence $\Phi_Q \colon \Db(X',\alpha') \to \Db(X',\alpha^{\prime -1})$ such that
  \[ \Phi^{\Hh}_Q = \id_{\Hh^0 \oplus \Hh^4} \oplus -\id_{\Hh^2} \colon \widetilde{\Hh}(X',B',\Z) \xrightarrow{\sim} \widetilde{\Hh}(X',-B',\Z). \]
\end{lem}
Recall that $\id_{\Hh^0 \oplus \Hh^4} \oplus -\id_{\Hh^2}$ is always non-signed by Example~\ref{exmp:non-signed_Hodge_isometry}.
\begin{proof}
  The Hodge isometry
  \[ \psi \colon \widetilde{\Hh}(X,B,\Z) \xrightarrow[\sim]{\Phi^{\Hh}_P} \widetilde{\Hh}(X',B',\Z) \xrightarrow[\sim]{\id_{\Hh^0 \oplus \Hh^4} \oplus -\id_{\Hh^2}} \widetilde{\Hh}(X',-B',\Z) \]
  is signed.
  By \cite[Thm.~0.1]{MR2310257}, we can thus find a Fourier--Mukai equivalence $\Phi_R \colon \Db(X,\alpha) \xrightarrow{\sim} \Db(X',\alpha^{\prime -1})$ with $\Phi^{\Hh}_R = \psi$.
  The composition $\Phi_Q \colonequals \Phi_R \circ (\Phi_P)^{-1}$ is another Fourier--Mukai equivalence \cite[p.~917]{MR2179782} whose induced action on cohomology is $\Phi^{\Hh}_Q = \id_{\Hh^0 \oplus \Hh^4} \oplus -\id_{\Hh^2}$.
\end{proof}
\begin{cor}\label{cor:trivial_det}
  It suffices to prove Theorem~\ref{thm:norhi} in the case where $\rk P_x = 0$ and $\det P_x \simeq \mathcal{O}_{X'}$ for all closed $x \in X$.
\end{cor}
\begin{proof}
  Arguing by contradiction, suppose $\Phi^{\Hh}_P$ is non-signed.
  By Lemma~\ref{lem:reduction-id}, we may assume that
  \[ \Phi^{\Hh}_P = \id_{\Hh^0 \oplus \Hh^4} \oplus -\id_{\Hh^2} \colon \widetilde{\Hh}(X',B',\Z) \xrightarrow{\sim} \widetilde{\Hh}(X',-B',\Z). \]
  For all $x \in X$, this implies $v^{-B'}(P_x) = \Phi^{\Hh}_P\bigl(v^{B'}\bigl(\kappa^\alpha(x)\bigr)\bigr) = \Phi^{\Hh}_P\bigl((0,0,1)\bigr) = (0,0,1)$, where $P_x$ is as in Definition~\ref{defn:FM_fiber} and $\kappa^\alpha(x)$ denotes the $\alpha$-twisted skyscraper sheaf at $x$.
  As $v^{-B'}(P_x) = \ch^{-B'}(P_x) \cdot \sqrt{\td(X)} = \ch^{-B'}(P_x) \cdot (1,0,1)$, we deduce $\ch^{-B'}(P_x)_0 = 0$ and $\ch^{-B'}(P_x)_2 = 0$.

  As in \S\ref{subsect:signed_deformation}, the $B$-field $-B'$ determines an $\alpha^{\prime -1}$-twisted, differentiable line bundle $L$ such that $\ch^{-B'}(P_x) = \ch\bigl(P_x \otimes L^{-1}\bigr)$.
  In particular, $\ch^{-B'}(P_x)_0 = \rk(P_x)$ and $\ch^{-B'}(P_x)_2 = \cc_1\bigl(P_x \otimes L^{-1}\bigr) = \cc_1\bigl(\det P_x \otimes L^{-\rk P_x}\bigr)$.
  Hence, $\rk P_x = 0$ and $\cc_1(\det P_x) = 0$.
  As $X'$ is a K3 surface, this yields $\det P_x \simeq \mathcal{O}_{X'}$.
\end{proof}

\subsection{The deformation family}\label{subsect:deformation-family}

We now use the moduli space of primitively polarized K3 surfaces with level structure to construct a deformation family for $(X',\alpha')$.
Fix $d \in \Z_{>0}$, $\ell = e_1 + d \cdot f_1 \in \Lambda$, and a marking $\theta' \colon \Hh^2(X',\Z) \xrightarrow{\sim} \Lambda$ such that $(\theta')^{-1}(\ell)$ is ample.
For the notation surrounding the K3 lattice $\Lambda$, see \S\ref{sect:intro}.
Using $\theta'$, we identify the $B$-field $B'$ with an element of $\Lambda_\Q$, which we still call $B'$.
For some $n \in \Z_{>0}$, the pair $(X',\theta')$ determines a point $o$ of the moduli space $M_d[n]$ of primitively $d$-polarized K3 surfaces with level-$n$ structure.
We briefly recall the construction of $M_d[n]$ and refer to \cite[Ch.~6]{huybrechts2016lectures} for details.

As in \S\ref{subsect:small-Picard-rank}, we use the period map to identify the moduli space of marked, primitively $d$-polarized K3 surfaces as
\[ N_d = D_d \smallsetminus \bigcup_{\substack{\delta \in \Lambda_d \\ (\delta .\delta) = -2}} \PP(\delta^\perp) \]
inside $\PP(\Lambda_{d\C})$, where $D_d = \lbrace p \in \PP(\Lambda_{d\C}) \suchthat (p.p) = 0,\, (p.\bar{p}) > 0 \rbrace$ is the period domain associated with $\Lambda_d$.
We can choose $n \in \Z_{>0}$ such that the period
\[ \per(B') \colonequals \min \{ k \in \Z_{>0} \suchthat kB' \in \Lambda \} \]
divides $n$ and the congruence subgroup
\[ \Gamma(n) \colonequals \left\{ g \in \Oo(\Lambda_d) \suchthat \bar{g} = \id \in \Oo\bigl(A_{\Lambda_d}\bigr) \text{ and } g \equiv \id{} \bmod n \right\} \subset \Oo(\Lambda_d) \]
is torsion-free \cite[Prop.~17.4]{MR0244260}.
Here, $A_{\Lambda_d}$ denotes the discriminant group of $\Lambda_d$ and $\bar{g}$ denotes the automorphism of $A_{\Lambda_d}$ induced by $g$.
In that case, the natural action of $\Gamma(n)$ on $N_d$ is free and properly discontinuous; cf.\ e.g.\ \cite[Rmk.~6.1.10, Prop.~6.1.12]{huybrechts2016lectures}.
The quotient
\[ M_d[n] \colonequals \Gamma(n) \setminus N_d \]
is the \emph{moduli space of primitively $d$-polarized K3 surfaces with level-$n$ structure}.
By the Baily--Borel theorem \cite{MR0216035}, it is a smooth, quasi-projective algebraic variety.
Moreover, it is equipped with a universal family $f' \colon \mathfrak{X}' \to M_d[n]$, a relatively ample line bundle on $\mathfrak{X}'$, and a universal level-$n$ structure $\theta \colon \Rr^2f'_* \mu_n \xrightarrow{\sim} \Lambda / n\Lambda$, which as before is compatible with all structures.

Since $\per(B') \mid n$, the element $B' \in \Lambda_\Q$ determines a class $\overline{B'} \in \tfrac{1}{n}\Lambda / \Lambda \simeq \Lambda / n\Lambda$.
Let $\mathfrak{b}' \colonequals \theta^{-1}\bigl(\overline{B'}\bigr) \in \Hh^0(M_d[n],\Rr^2f'_*\mu_n)$.
As
\[ (\Rr^2 f'_* \mu_n)_o = \underset{\substack{(S,s) \to (M_d[n],o) \\ \text{\'etale}}}\colim \Hh^2(\mathfrak{X}'_S,\mu_n), \]
there is an \'etale neighborhood $\rho \colon (S,s) \to (M_d[n],o)$ and $\beta'_S \in \Hh^2(\mathfrak{X}'_S,\mu_n)$ such that the image of $\beta'_S$ under the edge morphism $\Hh^2(\mathfrak{X}'_S, \mu_n) \to \Hh^0(S,\Rr^2 f'_{S,*} \mu_n)$ is $\rho^* \mathfrak{b}'$.
The inclusion $\mu_n \hookrightarrow \G_m$ induces a homomorphism $\chi\colon \Hh^2(\mathfrak{X}'_S,\mu_n) \to \Hh^2(\mathfrak{X}'_S,\G_m)$.
Set $\alpha'_S \colonequals \chi(\beta'_S) \in \Hh^2(\mathfrak{X}'_S,\G_m)$.

\subsection{Conclusion}\label{subsect:conclusion}

We remain in the setup of \S\ref{subsect:deformation-family}.
First, we construct infinitely many reduced, effective divisors $Z_m$ on $M_d[n]$.
Fix $k \in \Z_{\geq 0}$ such that $\bigl((B'+k\cdot f_2)\,.\,e_2\bigr) > 0$.
For each $m \in \Z_{>0}$, define
\[ B_m \colonequals (4m+3)d \cdot e_2 + \frac{1}{(4m+3)d} \cdot f_3. \]
Write $B'+k\cdot f_2-B_m = \eta\cdot \ell+\zeta_m$ for some $\eta \in \Q$ and $\zeta_m \in \Lambda_{d\Q}$.
Put
\[ \widetilde{Z}_m \colonequals D_d \cap \PP\bigl(\zeta^\perp_m\bigr) \subset D_d, \]
where the orthogonal complement is taken in $\Lambda_{d\C}$.
This is the period domain associated with $\zeta^\perp_m \cap \Lambda_d$.
The intersection $\widetilde{Z}_m \cap N_d$ parameterizes all marked, primitively $d$-polarized K3 surfaces whose rational Picard group contains $\zeta_m$;
in particular, $B'$ and $B_m$ induce the same Brauer class.

Let $Z_m$ be the image of $\widetilde{Z}_m \cap N_d$ in the quotient $M_d[n] = \Gamma(n) \setminus N_d$.
When $m \gg 0$, this is a so-called \textit{special divisor}; cf.\ \cite[\S\S 2,4]{MR3156424}.

\begin{lem}
  For $m \gg 0$, the $Z_m$ are reduced, effective divisors on $M_d[n]$, and infinitely many of them are pairwise distinct.
\end{lem}
\begin{proof}
  When $m \gg 0$,
  \[ (\zeta_m\,.\,\zeta_m) = (B'+k\cdot f_2-\eta\cdot \ell)^2 - 2(4m+3)d \cdot ((B'+k\cdot f_2)\,.\,e_2) - \frac{2}{(4m+3)d} \cdot (B'\,.\,f_3) \]
  is negative.
  In that case, the lattice $\zeta^\perp_m \cap \Lambda_d \subset \Lambda_d$ has signature $(2,18)$, so $\widetilde{Z}_m$ is a closed submanifold of $D_d$ of codimension $1$.

  On the other hand, let $\lambda_m$ denote the coefficient of $f_3$ in $\zeta_m$ and choose $\widehat{\zeta}_m \in \Lambda_d$ primitive such that $\Q \cdot \widehat{\zeta}_m = \Q \cdot \zeta_m$.
  By construction, $\lambda_m \neq 0$ for $m \gg 0$;
  moreover, $\bigl(\widehat{\zeta}_m\,.\,\widehat{\zeta}_m\bigr) \leq \lambda^{-2}_m (\zeta_m\,.\,\zeta_m) < -2$ because the sequence $(\lambda_m)$ is bounded and $\lim_{m \to \infty} (\zeta_m\,.\,\zeta_m) = -\infty$.
  Since any $\delta \in \Lambda_d$ with $(\delta\,.\,\delta) = -2$ is primitive as well, $\Q \cdot \widehat{\zeta}_m \neq \Q \cdot \delta$, and thus $\PP\bigl(\zeta^\perp_m\bigr)$ is not contained in $\PP\bigl(\delta^\perp\bigr)$.
  In particular,
  \[ \widetilde{Z}_m \cap N_d = \bigl(D_d \cap \PP\bigl(\zeta^\perp_m\bigr)\bigr) \smallsetminus \bigcup_{\substack{\delta \in \Lambda_d\\(\delta.\delta) =-2}} \PP\bigl(\delta^\perp\bigr) \]
  is a closed submanifold of $N_d$ of codimension $1$.

  Fix $m \gg 0$.
  The group generated by the reflections associated to $\gamma(\zeta_m)$, where $\gamma$ ranges over $\Gamma(n)$, is a discrete subgroup of $\Oo(\Lambda_{d\R})$;
  hence, it acts properly discontinuously on $N_d$.
  In other words, every $x \in N_d$ has an open neighborhood $V$ such that $s(V) \cap V = \varnothing$ for all but finitely many such reflections $s$.
  Therefore, $V$ meets only finitely many $N_d \cap \PP(\gamma(\zeta_m)^\perp)$.
  That is, the set of these hyperplane sections is locally finite.
  In particular, the orbit $\Gamma(n)\cdot \bigl(\widetilde{Z}_m \cap N_d\bigr)$ is closed in $N_d$ and $Z_m$ is a reduced, effective analytic divisor on $M_d[n]$.

  Let $\Gamma(n)_{\zeta_m}$ be the stabilizer subgroup of $\Gamma(n)$ with respect to $\zeta_m$.
  The map $\widetilde{Z}_m \cap N_d \to Z_m$ factors through the quotient
  \[ Z'_m \colonequals \Gamma(n)_{\zeta_m} \setminus \bigl(\widetilde{Z}_m \cap N_d\bigr). \]
  The connected components of $\widetilde{Z}_m$ are bounded symmetric domains of type IV, so $Z'_m$ is a smooth, quasi-projective variety \cite{MR0216035}.
  As the map $Z'_m \to M_d[n]$ is algebraic \cite[Thm.~3.10]{MR0338456}, its image $Z_m$ is constructible.
  Since Zariski and analytic closure coincide for constructible subsets, $Z_m$ is Zariski-closed.
  
  Lastly, $\lim_{m \to \infty} (\zeta_m\,.\,\zeta_m) = -\infty$ and the boundedness of $(\lambda_m)$ guarantee that the sequence $\bigl((\widehat{\zeta}_m\,.\,\widehat{\zeta}_m)\bigr)$ has a strictly decreasing subsequence $\bigl((\widehat{\zeta}_{m_j}\,.\,\widehat{\zeta}_{m_j})\bigr)$.
  The corresponding orbits $\Gamma(n) \cdot \bigl(\widetilde{Z}_{m_j} \cap N_d\bigr)$, and hence the divisors $Z_{m_j}$, must be pairwise distinct.
\end{proof}

\begin{proof}[Proof of Theorem~\ref{thm:norhi}]
  By Corollary~\ref{cor:trivial_det}, we may assume that $\rk P_x = 0$ and $\det P_x \simeq \mathcal{O}_{X'}$ for all closed $x\in X$.
  We use the deformation family constructed in \S\ref{subsect:deformation-family}.
  By Proposition~\ref{prop:deformation} applied to $\mathfrak{X}'_S \to S$, we can refine $\rho$ by a further connected \'etale cover $(U,u) \to (S,s)$ and obtain
  \begin{enumerate}[label={\upshape(\roman*)}]
    \item a family of K3 surfaces $\mathfrak{X}_U \to U$ with $(\mathfrak{X}_U)_u = X$,
    \item a torsion class $\alpha^{-1}_U \in \Hh^2(\mathfrak{X}_U,\G_m)$ with $(\alpha^{-1}_U)_u = \alpha^{-1}$, and
    \item a kernel $\mathfrak{P} \in \Db(\mathfrak{X}_U \times \mathfrak{X}'_U, \alpha^{-1}_U \boxtimes \alpha'_U)$ (where $\mathfrak{X}'_U$ and $\alpha'_U$ are the base changes of $\mathfrak{X}'_S$ and $\alpha'_S$ to $U$)  with $\mathfrak{P}_u = P$ which induces an equivalence on each fiber.
  \end{enumerate}
  We now prove Theorem~\ref{thm:norhi} in two steps.
  \begin{description}[style=unboxed,leftmargin=*]
    \item[Step 1]
      We make the additional assumption that the twisted derived category of a K3 surface which corresponds to a very general point of $M_d[n]$ (with twist induced by $\mathfrak{b}'$) does not contain any spherical objects.
      Since the image of $\rho \colon (U,u) \to (M_d[n],o)$ is open, we can find $t \in U$ for which $\Db\bigl((\mathfrak{X}'_U)_t,(\alpha'_U)_t\bigr)$ has that property.
      As $\Phi_{\mathfrak{P}_t}$ is an equivalence, $\Db\bigl((\mathfrak{X}_U)_t,(\alpha_U)_t\bigr)$ does not contain any spherical objects either.
      By \cite[Cor.~3.19]{MR2388559}, $\Phi^{\Hh}_{\mathfrak{P}_t}$ is signed.
      Thus, Proposition~\ref{prop:signed_deform} shows that $\Phi^{\Hh}_P = \Phi^{\Hh}_{\mathfrak{P}_u}$ must also be signed.
    \item[Step 2]
      We reduce the general case to that of Step 1.
      As the union of infinitely many distinct, reduced, effective divisors is Zariski-dense and $\rho(U)$ is Zariski-open, we can choose $z \in U$ and $m \in \Z_{>0}$ such that $\rho(z) \in Z_m$.
      In particular, we can find a marking $\theta_z$ on the K3 surface $(\mathfrak{X}'_U)_z$ for which $\theta^{-1}_z(B'+k\cdot f_2-B_m) \in \Pic\bigl((\mathfrak{X}'_U)_z\bigr) \otimes_\Z \Q$.
      Hence, the $B$-fields $\theta^{-1}_z(B')$ and $\theta^{-1}_z(B_m)$ induce the same Brauer class on $(\mathfrak{X}'_U)_z$.

      On the other hand, whether an isometry between twisted Hodge structures is signed is independent of the chosen $B$-field lifts.
      By Proposition~\ref{prop:signed_deform}, $\Phi^{\Hh}_P = \Phi^{\Hh}_{\mathfrak{P}_u}$ is signed if and only if $\Phi^{\Hh}_{\mathfrak{P}_z}$ is so.
      It therefore suffices to prove Theorem~\ref{thm:norhi} for the Fourier--Mukai kernel $\mathfrak{P}_z$ and the $B$-field $B_m$ in place of $P$ and $B'$.

      As $(B_m\,.\,\ell) = (B_m\,.\,B_m) = 0$, the discussion before Lemma~\ref{lem:inters_form_LBd} shows that the $(1,1)$-part of the $B_m$-twisted Hodge structure of a very general marked, $d$-polarized K3 surface is
      \[ L^{B_m}_d = \begin{pmatrix} 0 & 0 & (4m+3)d \\ 0 & 2d & 0 \\ (4m+3)d & 0 & 0 \end{pmatrix}_{\textstyle \raisebox{3pt}{.}} \]
      But $L^{B_m}_d$ does not represent $-2$:
      when $d \geq 2$, use that the intersection form takes values in $2d\Z$, and when $d = 1$, observe that $-1$ is not a square modulo $4m+3$ by the quadratic reciprocity law for a prime factor which is $3$ modulo $4$.
      Since the Mukai vector of a spherical object is a $(-2)$-class, the $B_m$-twisted derived category of a very general point of $M_d[n]$ cannot contain any spherical objects. \qedhere
  \end{description}
\end{proof}

\begin{acknowledgements}
I would like to thank Daniel Huybrechts for suggesting that I think about these questions and for offering invaluable mathematical support and encouragement.
I am also grateful to Bhargav Bhatt for his indispensable help with various aspects of this paper.
Further thanks go to Yajnaseni Dutta, Takumi Murayama, Matt Stevenson, and two anonymous referees for long lists of comments on earlier drafts.
Lastly, I benefited from helpful conversations with Brandon Carter, Igor Dolgachev, Aron Heleodoro, Max Lieblich, Siddharth Mathur, Rainer Schulze-Pillot, Bertrand To\"en, and Gabriele Vezzosi.
\end{acknowledgements}


\begin{thebibliography}{HLOY04}

  \bibitem[ACV03]{MR2007376}
    D.~Abramovich, A.~Corti, and A.~Vistoli.
    Twisted bundles and admissible covers.
    \textit{Comm.\ Algebra}, 31(8):3547--3618, 2003.
    Special issue in honor of Steven L. Kleiman.

  \bibitem[BB66]{MR0216035}
    W.~Baily and A.~Borel.
    Compactification of arithmetic quotients of bounded symmetric domains.
    \textit{Ann.\ of Math.\ (2)}, 84(3):442--528, 1966.

  \bibitem[Bor69]{MR0244260}
    A.~Borel.
    \textit{Introduction aux groupes arithm\'etiques}.
    Publications de l'Institut de Math\'ematique de l'Universit\'e de Strasbourg.
    Actualit\'es Scientifiques et Industrielles, No.~1341.
    Paris: Hermann \& Cie., 1969.

  \bibitem[Bor72]{MR0338456}
    A.~Borel.
    Some metric properties of arithmetic quotients of symmetric spaces and an extension theorem.
    \textit{J.\ Differential Geometry}, 6(4):543--560, 1972.
    Collection of articles dedicated to S.~S.~Chern and D.~C.~Spencer on their sixtieth birthdays.

  \bibitem[Bor86]{MR849050}
    C.~Borcea.
    Diffeomorphisms of a K3 surface.
    \textit{Math.\ Ann.}, 275(1):1--4, 1986.

  \bibitem[C\u{a}l00]{MR2700538}
    A.~C\u{a}ld\u{a}raru.
    \textit{Derived categories of twisted sheaves on Calabi--Yau manifolds}.
    Ann Arbor: ProQuest, 2000.
    Ph.D.\ thesis, Cornell University.

  \bibitem[Cas78]{MR522835}
    J.~Cassels.
    \textit{Rational quadratic forms}, volume 13 of \textit{London Mathematical Society Monographs}.
    London: Academic Press, 1978.

  \bibitem[CS07]{MR2329310}
    A.~Canonaco and P.~Stellari.
    Twisted Fourier--Mukai functors.
    \textit{Adv.\ Math.}, 212(2):484--503, 2007.

  \bibitem[Don90]{MR1066174}
    S.~Donaldson.
    Polynomial invariants for smooth four-manifolds.
    \textit{Topology}, 29(3):257--315, 1990.

  \bibitem[Gir71]{MR0344253}
    J.~Giraud.
    \textit{Cohomologie non ab\'elienne}, volume 179 of \textit{Die Grundlehren der mathematischen Wissenschaften}.
    Berlin: Springer, 1971.

  \bibitem[HLOY04]{MR2047679}
    S.~Hosono, B.~Lian, K.~Oguiso, and S.-T.~Yau.
    Autoequivalences of derived category of a K3 surface and monodromy transformations.
    \textit{J.\ Algebraic Geom.}, 13(3):513--545, 2004.

  \bibitem[HMS08]{MR2388559}
    D.~Huybrechts, E.~Macr\`i, and P.~Stellari.
    Stability conditions for generic K3 categories.
    \textit{Compos.\ Math.}, 144(1):134--162, 2008.

  \bibitem[HMS09]{MR2553878}
    D.~Huybrechts, E.~Macr\`i, and P.~Stellari.
    Derived equivalences of K3 surfaces and orientation.
    \textit{Duke Math.\ J.}, 149(3):461--507, 2009.

  \bibitem[HR14]{MR3148551}
    J.~Hall and D.~Rydh.
    The Hilbert stack.
    \textit{Adv.\ Math.}, 253:194--233, 2014.

  \bibitem[HS05]{MR2179782}
    D.~Huybrechts and P.~Stellari.
    Equivalences of twisted K3 surfaces.
    \textit{Math.\ Ann.}, 332(4):901--936, 2005.

  \bibitem[HS06]{MR2310257}
    D.~Huybrechts and P.~Stellari.
    Proof of C\u{a}ld\u{a}raru's conjecture.
    Appendix to ``Moduli spaces of twisted sheaves on a projective variety'' by K.~Yoshioka.
    In \textit{Moduli spaces and arithmetic geometry}, volume 45 of \textit{Adv.\ Stud.\ Pure Math.}, pp.~31--42.
    Tokyo: Math.\ Soc.\ Japan, 2006.

  \bibitem[Huy16]{huybrechts2016lectures}
    D.~Huybrechts.
    \textit{Lectures on K3 Surfaces}, volume 158 of \textit{Cambridge Studies in Advanced Mathematics}.
    Cambridge: Cambridge University Press, 2016.

  \bibitem[Huy17]{huybrechts2017cubic}
    D.~Huybrechts.
    The K3 category of a cubic fourfold.
    \textit{Compos.\ Math.}, 153(3):586--620, 2017.

  \bibitem[KM76]{MR0437541}
    F.~Knudsen and D.~Mumford.
    The projectivity of the moduli space of stable curves. I. Preliminaries on ``det'' and ``Div''.
    \textit{Math.\ Scand.}, 39(1):19--55, 1976.

  \bibitem[Kne56a]{MR0082514}
    M.~Kneser.
    Klassenzahlen indefiniter quadratischer Formen in drei oder mehr Ver\"anderlichen.
    \textit{Arch.\ Math.\ (Basel)}, 7:323--332, 1956.

  \bibitem[Kne56b]{MR0080101}
    M.~Kneser.
    Orthogonale Gruppen \"uber algebraischen Zahlk\"orpern.
    \textit{J.\ Reine Angew.\ Math.}, 196:213--220, 1956.
 
  \bibitem[Kov94]{MR1314742}
    S.~Kov\'acs.
    The cone of curves of a K3 surface.
    \textit{Math.\ Ann.}, 300(4):681--691, 1994.

  \bibitem[KS60]{MR0115189}
    K.~Kodaira and D.~Spencer.
    On deformations of complex analytic structures. III. Stability theorems for complex structures.
    \textit{Ann. of Math. (2)}, 71(1):43--76, 1960.

  \bibitem[Kud13]{MR3156424}
    S.~Kudla.
    A note about special cycles on moduli spaces of K3 surfaces.
    In \textit{Arithmetic and geometry of K3 surfaces and Calabi--Yau threefolds}, volume 67 of \textit{Fields Inst.\ Commun.}, pp.~411--427.
    New York: Springer, 2013.

  \bibitem[Lie07]{MR2309155}
    M.~Lieblich.
    Moduli of twisted sheaves.
    \textit{Duke Math.\ J.}, 138(1):23--118, 2007.

  \bibitem[Lie08]{MR2388554}
    M.~Lieblich.
    Twisted sheaves and the period-index problem.
    \textit{Compos.\ Math.}, 144(1):1--31, 2008.

  \bibitem[LO15]{MR3429474}
    M.~Lieblich and M.~Olsson.
    Fourier--Mukai partners of K3 surfaces in positive characteristic.
    \textit{Ann.\ Sci.\ \'Ec.\ Norm.\ Sup\'er.\ (4)}, 48(5):1001--1033, 2015.

  \bibitem[MM09]{miranda2009embeddings}
    R.~Miranda and D.~Morrison.
    \textit{Embeddings of integral quadratic forms}.
    \url{http://web.math.ucsb.edu/~drm/manuscripts/eiqf.pdf}, 2009.

  \bibitem[Muk87]{MR893604}
    S.~Mukai.
    On the moduli space of bundles on K3 surfaces. I.
    In \textit{Vector bundles on algebraic varieties (Bombay, 1984)}, volume 11 of \textit{Tata Inst.\ Fund.\ Res.\ Stud.\ Math.}, pp.~341--413.
    Bombay: Tata Inst.\ Fund.\ Res., 1987.

  \bibitem[Nik79]{MR525944}
    V.~Nikulin.
    Integer symmetric bilinear forms and some of their geometric applications.
    \textit{Izv.\ Akad.\ Nauk SSSR Ser.\ Mat.}, 43(1):111--177, 1979.

  \bibitem[Ols06]{MR2206635}
    M.~Olsson.
    Deformation theory of representable morphisms of algebraic stacks.
    \textit{Math.\ Z.}, 253(1):25--62, 2006.

  \bibitem[OM00]{MR1754311}
    T.~O'Meara.
    \textit{Introduction to quadratic forms}.
    Reprint of the 1973 edition.
    \textit{Classics in Mathematics}.
    Berlin: Springer-Verlag, 2000.

  \bibitem[Orl97]{MR1465519}
    D.~Orlov.
    Equivalences of derived categories and K3 surfaces.
    \textit{J.\ Math.\ Sci.\ (New York)}, 84(5):1361--1381, 1997.

  \bibitem[Plo05]{ploog-thesis}
    D.~Ploog.
    \textit{Groups of autoequivalences of derived categories of smooth projective varieties}.
    Berlin: Logos Verlag, 2005.
    Ph.D.\ thesis, FU Berlin.

  \bibitem[Riz06]{MR2263236}
    J.~Rizov.
    Moduli stacks of polarized K3 surfaces in mixed characteristic.
    \textit{Serdica Math.\ J.}, 32(2-3):131--178, 2006.

  \bibitem[SGA3]{MR0274458}
    M.~Demazure and A.~Grothendieck, eds.
    \textit{Sch\'emas en groupes. I: Propri\'et\'es g\'en\'erales des sch\'emas en groupes}.
    S\'eminaire de G\'eom\'etrie Alg\'ebrique du Bois Marie 1962--64 (SGA 3).
    Volume 151 of \textit{Lecture Notes in Mathematics}.
    Berlin: Springer, 1970.

  \bibitem[SGA6]{SGA6}
    P.~Berthelot, A.~Grothendieck, and L.~Illusie, eds.
    \textit{Th\'eorie des intersections et th\'eor\`eme de Riemann--Roch}.
    S\'eminaire de G\'eom\'etrie Alg\'ebrique du Bois Marie 1966--67 (SGA 6).
    Volume 225 of \textit{Lecture Notes in Mathematics}. 
    Berlin: Springer, 2006.

  \bibitem[SP17]{stacks-project}
    The Stacks Project Authors.
    \textit{Stacks Project}.
    \url{http://stacks.math.columbia.edu}, 2017.

  \bibitem[STV15]{MR3341464}
    T.~Sch\"urg, B.~To\"en, and G.~Vezzosi
    Derived algebraic geometry, determinants of perfect complexes, and applications to obstruction theories for maps and complexes.
    \textit{J.\ Reine Angew.\ Math.}, 702:1--40, 2015.

  \bibitem[Sze01]{MR1866907}
    B.~Szendr\H{o}i.
    Diffeomorphisms and families of Fourier--Mukai transforms in mirror symmetry.
    In \textit{Applications of algebraic geometry to coding theory, physics and computation (Eilat, 2001)}, volume 36 of \textit{NATO Sci.\ Ser.\ II Math.\ Phys.\ Chem.}, pp.~317--337.
    Dordrecht: Kluwer Acad.\ Publ., 2001.

  \bibitem[To\"e99]{toen-thesis}
    B.~To\"en.
    \textit{K-th\'eorie et cohomologie des champs alg\'ebriques}.
    1999.
    Ph.D.\ thesis, Universit\'e Paul Sabatier-Toulouse III.

  \bibitem[To\"e12]{MR2957304}
    B.~To\"en.
    Derived Azumaya algebras and generators for twisted derived categories.
    \textit{Invent.\ Math.}, 189(3):581--652, 2012.

  \bibitem[TV07]{MR2493386}
    B.~To\"en and M.~Vaqui\'e.
    Moduli of objects in dg-categories.
    \textit{Ann.\ Sci.\ \'Ecole Norm.\ Sup.\ (4)}, 40(3):387--444, 2007.

\end{thebibliography}
\end{document}